\documentclass[]{interact}

\usepackage{epstopdf}
\usepackage[caption=false]{subfig}

\usepackage[numbers,sort&compress]{natbib}
\usepackage{algorithm}
\usepackage{color}
\usepackage{algpseudocode}
\usepackage{mathtools}

\usepackage{enumitem}

\algnewcommand\algorithmicparameter{\textbf{Input:}}
\algnewcommand\Parameter{\item[\algorithmicparameter]}
\algnewcommand\algorithmicoutput{\textbf{Output:}}
\algnewcommand\Output{\item[\algorithmicoutput]}

\theoremstyle{plain}
\newtheorem{theorem}{Theorem}[section]
\newtheorem{lemma}[theorem]{Lemma}
\newtheorem{corollary}[theorem]{Corollary}

\newtheorem{assumption}[theorem]{Assumption}

\theoremstyle{definition}

\theoremstyle{remark}
\newtheorem{remark}{Remark}

\usepackage{xparse}
\NewDocumentCommand{\E}{ O{} m }{\mathbb{E} \left[ \ifthenelse{\equal{#1}{}}{}{\left.} #2
  \ifthenelse{\equal{#1}{}}{}{\right|#1} \right]}

\newcommand{\F}{\mathcal{F}}
\renewcommand{\O}{\mathcal{O}}
\newcommand{\R}{\mathbb{R}}
\newcommand{\prox}[2]{\mathrm{prox}_{#1}\left(#2\right)}
\newcommand{\dist}[1]{\mathrm{dist}\left(#1\right)}
\newcommand{\sdist}[1]{\mathrm{dist}^2\left(#1\right)}

\DeclareMathOperator*{\argmin}{arg\,min}
\newcommand{\ip}[1]{\left\langle #1 \right\rangle}

\newcommand\numberthis{\addtocounter{equation}{1}\tag{\theequation}}

\begin{document}


\title{Accelerated Dual-Averaging Primal-Dual Method for Composite Convex Minimization\footnote{Dedicated to Professor Ya-xiang Yuan on the occasion of his 60th birthday.}}

\author{
\name{Conghui Tan\textsuperscript{a}, Yuqiu Qian\textsuperscript{b},
 Shiqian Ma\textsuperscript{c} \thanks{Corresponding Author. Email: sqma@ucdavis.edu} and Tong Zhang\textsuperscript{d}}
\affil{\textsuperscript{a}WeBank AI Group, Shenzhen, China; \textsuperscript{b}Tencent Inc., Shenzhen, China;
\\ \textsuperscript{c}Department of Mathematics, University of California, Davis; \textsuperscript{d}Departments of Computer Science and Mathematics, The Hong Kong University of Science and Technology}
}

\maketitle

\begin{abstract}
Dual averaging-type methods are widely used in industrial machine learning applications due to their ability to promoting solution structure (e.g., sparsity) efficiently. In this paper, we propose a novel accelerated dual-averaging primal-dual algorithm for minimizing a composite convex function.
We also derive a stochastic version of the proposed method which solves empirical risk minimization, and its advantages on handling sparse data are demonstrated both theoretically and empirically.
\end{abstract}


\section{Introduction}
\label{sec:intro}

In this paper, we consider minimizing the following composite convex function:
\begin{equation}
\label{eq:problem}
\min_{x\in\R^d}\left\{ P(x)\coloneqq f(Ax) + g(x)\right\},
\end{equation}
where $A\in\R^{n\times d}$, and both $f:\R^n\rightarrow \R\cup\{+\infty\}$ and $g: \R^d\rightarrow \R\cup\{+\infty\}$ are convex closed functions. Here $f$ can be either smooth or non-smooth, and we assume $g$ has easy proximal mapping. Problem \eqref{eq:problem} covers a wide range of applications. For example, choosing $f$ to be the indicator function of a convex set $C=\{z\in\R^n|z\leq b\}$ corresponds to minimizing a convex function over a polyhedron. It covers the Lasso problem \cite{Tibshirani-LASSO-1996}
\begin{equation}
\label{eq:lasso}
\min_{x\in\R^d}\left\{ \frac{1}{2n}\|Ax-b\|^2_2 + \lambda\|x\|_1\right\},
\end{equation}
by setting $f(u)=\frac{1}{2n}\|u-b\|^2_2$ and $g(x)=\lambda\|x\|_1$. Another application of the form \eqref{eq:problem} is the support vector machine (SVM):
\begin{equation}
\label{eq:svm}
\min_{x\in\R^d}\frac{1}{n}\sum_{i=1}^n \max\left\{1-\ip{b_i a_i, x},\, 0\right\} + \frac{\lambda}{2}\|x\|^2_2,
\end{equation}
where $a_i\in\R^d$ is the feature vector of the $i$-th data sample, and $b_i\in\{\pm 1\}$ is the corresponding label. 

For smooth $f$, a classical way to solve \eqref{eq:problem} is the proximal gradient method (PGM) and its accelerations \cite{beck2009fast,tseng2008accelerated}. PGM for solving \eqref{eq:problem} iterates as
\[
x^{t+1}=\prox{\eta g}{x^t - \eta A^\top\nabla f(Ax^t)},
\]
where $\eta>0$ is the step size.
{Dual averaging (DA, \cite{nesterov2009primal}) algorithm} is another widely used algorithm for solving \eqref{eq:problem}, which iterates as
\[
x^{t+1}=\prox{\sum_{k=0}^t\beta_t g}{x^0 - \sum_{k=0}^t \beta_k A^\top\nabla f(Ax^k)},
\]
where $\{\beta_t\}$ are the step sizes. Different from PGM, in each iteration, DA always starts at the initial iterate $x^0$, averages all the past gradients, and then conducts proximal mapping. Dual-averaging type methods are widely used in many industrial machine learning applications due to the following advantages over PGM \cite{mcmahan2013ad, gao2017learning, mcmahan2017survey}. First, it is observed that DA is better in promoting solution structure (e.g., sparsity) than PGM \cite{xiao2010dual, mcmahan2011follow}. Second, DA can deal with sparse data much more efficiently than PGM. We will provide more details in Section \ref{sec:sparse}.

{In this paper, we develop a new dual-averaging primal-dual (DAPD) method for solving \eqref{eq:problem}, which has accelerated optimal convergence rate. When $f(Ax)$ has a finite-sum structure, we develop a stochastic version of DAPD, named SDAPD, which is also optimal, and has better overall complexity on sparse data comparing with existing algorithms of the same type.} 


{\bf Notation.} The following notation is adopted throughout this paper. For the matrix $A\in\R^{n\times d}$ used in \eqref{eq:problem}, we use $a_i^\top$ to denote the $i$-th row of $A$ and $a_{ij}$ to denote the $j$-th coordinate of $a_i$ ($1\leq i\leq n$, $1\leq j\leq d$). We define 
\begin{equation}\label{eq:norm_A}
R\coloneqq \|A\|_2
\quad\text{and}\quad
\bar{R}\coloneqq \max_{i=1,\dots,n} \|a_i\|_2.
\end{equation}
Note that $\|z\|_2$ denotes the spectral norm if $z$ is a matrix, and $\ell_2$ norm if $z$ is a vector. It is easy to show that $R$ and $\bar{R}$ have the following relationship: $\bar{R} \leq R \leq \sqrt{n}\bar{R}$. We use $\rho$ to denote the proportion of non-zero entries in $A$ (note $0<\rho\leq 1$). To ease the later discussion on computational complexity, without loss of generality, we assume $\rho\geq 1/n$ and $\rho\geq 1/d$, which happens for large-scale problems. For a convex set $C$, $\dist{x,C}\coloneqq \inf_{x'\in C}\|x - x'\|_2$ is the distance between point $x$ and set $C$. For any function $h(u):\R^p\rightarrow \R$, its proximal mapping is defined as:
\[
  \prox{h}{u}\coloneqq\argmin_{v\in\R^p}\left\{h(v) + \frac{1}{2}\|v - u\|^2_2\right\}, \quad \forall u\in\R^p.
\]
The domain of function $h(u)$ is denoted as $\mathrm{dom}\, h\coloneqq\{u\in\R^p| h(u)<+\infty\}$ and its conjugate function is defined as $h^*(v) = \sup_{u\in\R^p}\left\{\ip{v,u} - h(u)\right\}$. $\partial h(u)$ denotes the subdifferential of $h$ at $u$. The function $h(u)$ is said to be $\mu$-strongly convex if
\[
h(v)\geq h(u) + \ip{s,v-u} + \frac{\mu}{2}\|v-u\|^2_2,\quad\forall s\in\partial h(u),\ u,v\in\R^p.
\]
$h(u)$ is called $L$-Lipschitz continuous if it satisfies
\[
|h(u) - h(v)|\leq L\|u - v\|_2,\quad\forall u,v\in\R^p.
\]
$h(u)$ is called $\zeta$-smooth  if it is differentiable and its gradient is $\zeta$-Lipschitz continuous, i.e.,
\[
\vspace{-0.2cm}
\|\nabla h(u) - \nabla h(v)\|_2\leq \zeta\|u - v\|_2
,\quad\forall u,v\in\R^p.
\]

\section{The Dual-Averaging Primal-Dual Algorithm}\label{sec:algorithm}

In this section, we present our dual-averaging primal-dual (DAPD) algorithm, which solves the following primal-dual formulation of problem \eqref{eq:problem}:
\begin{equation}\label{eq:primal_dual}
\min_{x\in\R^d}\max_{y\in\R^n} \left\{F(x,y)\coloneqq  g(x) + \ip{y,Ax} - f^*(y)\right\}.
\end{equation}
We use $(x^*,y^*)$ to denote a pair of optimal primal-dual solutions to \eqref{eq:primal_dual}, and $X^*$ and $Y^*$ the sets containing all optimal primal and dual solutions, respectively.

\begin{algorithm}[t]
    \caption{Dual-Averaging Primal-Dual (DAPD) Method}\label{alg:1}
    \begin{algorithmic}[1]
        \Parameter initial points $x^0$ and $y^0$, primal and dual step sizes $\{\beta_t\}$, $\{\eta_t\}$ and $\{\tau_t\}$
        \State Initialize $B_0=\beta_0$
        \For{$t=0,1,2,\dots$}
            \State Compute intermediate variable:
                \begin{equation}
                    \label{eq:alg1_x_int}
                    \bar{x}^{t+1} \coloneqq \prox{\eta_t g}{x^t - \eta_t A^\top y^t}
                \end{equation}
            \State Update dual variable:
                \begin{equation}
                    \label{eq:alg1_y}
                    y^{t+1} \coloneqq \prox{\tau_t f^*}{y^t + \tau_t A\bar{x}^{t+1}}
                \end{equation}
            \State Update primal variable via a dual-averaged step:
                \begin{equation}
                \label{eq:alg1_x}
                x^{t+1} \coloneqq \prox{B_t g}{x^0 - \sum_{k=0}^t\beta_k A^\top y^{k+1}}
                \end{equation}
            \State Update $B_{t+1}\coloneqq B_t+\beta_{t+1}$
        \EndFor
    \end{algorithmic}
\end{algorithm}

The details of DAPD algorithm are presented in Algorithm \ref{alg:1}. In each iteration, DAPD first conducts one primal proximal gradient step
to compute the intermediate variable $\bar{x}^{t+1}$, and then $y^{t+1}$ is computed using the gradient evaluated at $\bar{x}^{t+1}$. Finally, $x^{t+1}$ is updated in \eqref{eq:alg1_x}, which adopts a dual-averaging type update rule. Note that all the past dual intermediate variables $\{y^{k+1}\}_{k=0}^t$ play a role here, and the gradient used in \eqref{eq:alg1_x} is a weighted sum of them, instead of simply $y^{t+1}$. 
The update of $y^{t+1}$ in \eqref{eq:alg1_y} can be viewed as an extragradient step \cite{korpelevich1976extragradient, korpelevichextrapolation}, since the gradient used here is evaluated at the intermediate variable $\bar{x}^{t+1}$ instead of $x^t$. 
Moreover, \eqref{eq:alg1_y} and \eqref{eq:alg1_x} have the flavor of the primal-dual hybrid gradient \cite{chambolle2011first}. 
Note that the step size used in the proximal mapping in \eqref{eq:alg1_x} is $B_t$, which is much larger than $\beta_t$. 
This helps promote the desired structures of solution $x^t$. For instance, if $g$ is the $\ell_1$ norm, then $x^{t+1}$ generated by \eqref{eq:alg1_x} is more likely to be sparse because $B_t$ is large. 

When implementing DAPD, the summation in \eqref{eq:alg1_x} needs to be computed incrementally. By doing so, the main computation cost in each iteration of Algorithm \ref{alg:1} lies in the matrix-vector multiplications $A^\top y^t$, $A\bar{x}^{t+1}$ and $A^\top y^{t+1}$. Since $A$ is a $n$-by-$d$ matrix with sparsity $\rho$, these multiplications can be done in $\O(\rho nd)$ operations. 

We now analyze the convergence rate of DAPD (Algorithm \ref{alg:1}). The following assumption is made throughout this section.
\begin{assumption}
$f$ is $(1/\gamma)$-smooth ($\gamma\geq 0$), and $g(x)$ is $\mu$-strongly convex ($\mu\geq 0$).
\end{assumption}
Note that $\gamma=0$ means that $f$ is non-smooth, and $\mu=0$ means that $g$ is non-strongly convex.

{
Although some parts of our DAPD algorithm look very similar to the primal-dual hybrid gradient (PDHG, \citep{chambolle2011first}), technical challenges still exist if we want to directly adapt the analysis of PDHG to our algorithm.
\begin{enumerate}[label=(\roman*)]
\item $\bar{x}^{t+1}$ in DAPD is obtained by a gradient step instead of extrapolation step. In the analysis of PDHG, the extrapolation step plays an important role in canceling the mismatch between primal and dual variables. Here we need a new approach to tackle this difficulty.
\item Since the primal updates consist of two gradient descent steps, two very different sequences of primal step sizes $\{\eta_t\}$ and $\{\beta_t\}$ and the dual step size $\{\tau_t\}$ need to be specified. This requires us to carefully balance these three parameters so that we can obtain the fastest convergence.
\item The update of $x^{t+1}$ in DAPD is in the dual averaging style, which is very different from PDHG in that it involves all the past gradients rather than simply the gradient at $y^{t+1}$. This makes it difficult to relate this step to the objective function value $F(x^{t+1},y^{t+1})$.
\end{enumerate}
In order to tackle these issues, new techniques are needed for the analysis. We define a potential function $\phi_t$ to characterize the dual-averaging step as follows:
}
\begin{equation}\label{eq:def_phi}
    \phi_{t}(x)\coloneqq \frac{1}{2}\left\|x-x^0\right\|^2_2 + \sum_{k=0}^{t-1} \beta_k\left(g(x) + \langle y^{k+1}, Ax\rangle \right).
\end{equation}
From \eqref{eq:alg1_x} it is easy to observe that $x^{t+1} \coloneqq\argmin_x \phi_{t+1}(x)$. Besides, since $g(x)$ is $\mu$-strongly convex,
$\phi_t(x)$ is strongly convex with strong convexity parameter $1+\sum_{k=0}^{t-1}\beta_k\mu = 1+B_{t-1}\mu$. Moreover, we denote $\phi_t^*\coloneqq\min_{x\in\R^d} \phi_t(x)$.


The following lemma characterizes the change of $\phi_t^*$ after one iteration.
\begin{lemma}\label{lemma:x}
Assume
\begin{equation}\label{eq:condition_beta}
    \eta_t(1+B_{t-1}\mu)\geq\beta_t.
\end{equation}
We have
\begin{align*}
\phi_{t+1}^* - \phi_t^* \geq \beta_t\left(g(\bar{x}^{t+1}) + \langle y^{t+1}, A\bar{x}^{t+1}\rangle \right)
    - \frac{\beta_t R^2\eta_t}{2}\|y^{t+1} - y^t\|^2_2.
\numberthis \label{eq:change_phi}
\end{align*}
\end{lemma}

\begin{proof}
From the strong convexity of $\phi_{t+1}(x)$ and \eqref{eq:condition_beta}, we obtain
\begin{align*}
    \phi_{t+1}^* =  & \phi_{t+1}(x^{t+1}) = \phi_t(x^{t+1}) + \beta_t\left(g(x^{t+1}) + \langle y^{t+1}, Ax^{t+1}\rangle \right) \\
               \geq & \phi_t^* + \frac{1+B_{t-1}\mu}{2}\|x^t - x^{t+1}\|^2_2 + \beta_t\left(g(x^{t+1}) + \langle y^{t+1}, Ax^{t+1}\rangle \right) \\
               \geq & \phi_t^* + \frac{\beta_t}{2\eta_t}\|x^t - x^{t+1}\|^2_2 + \beta_t\left(g(x^{t+1}) + \langle y^{t+1}, Ax^{t+1}\rangle \right).
\numberthis \label{eq:x_phi_star}
\end{align*}
Note that \eqref{eq:alg1_x_int} can be rewritten as $\bar{x}^{t+1} = x^t - \eta_t\left(A^\top y^t + s \right),\ \exists s\in\partial g(\bar{x}^{t+1})$,
which yields
\begin{align*}
&\|x^t - x^{t+1}\|^2_2 - \|x^t - \bar{x}^{t+1}\|^2_2 - \|\bar{x}^{t+1} - x^{t+1}\|^2_2 \\
=& 2\langle x^t - \bar{x}^{t+1}, \bar{x}^{t+1} - x^{t+1} \rangle = 2\eta_t\langle A^\top y^t + s, \bar{x}^{t+1} - x^{t+1} \rangle \\
\geq& 2\eta_t\left(\langle y^t, A(\bar{x}^{t+1} - x^{t+1}) \rangle +  g(\bar{x}^{t+1}) - g(x^{t+1})\right),
\numberthis \label{eq:derive_x_int}
\end{align*}
where the inequality is due to the convexity of $g(x)$. Combining \eqref{eq:x_phi_star} and \eqref{eq:derive_x_int} yields
\begin{align*}
& \phi_{t+1}^* \\
\geq& \phi_t^* + \frac{\beta_t}{2\eta_t}\left[\|x^t - \bar{x}^{t+1}\|^2_2 + \|\bar{x}^{t+1} - x^{t+1}\|^2_2
        + 2\eta_t\left(\langle y^t, A(\bar{x}^{t+1} - x^{t+1}) \rangle
        + g(\bar{x}^{t+1}) - g(x^{t+1})\right) \right] \\
    &+ \beta_t\left(g(x^{t+1}) + \langle y^{t+1}, Ax^{t+1}\rangle \right) \\
=& \phi_t^* + \frac{\beta_t}{2\eta_t}\left(\|x^t - \bar{x}^{t+1}\|^2_2 + \|\bar{x}^{t+1} - x^{t+1}\|^2_2\right)
    + \beta_t\left(g(\bar{x}^{t+1}) + \langle y^{t+1}, A\bar{x}^{t+1}\rangle \right) \\
&   + \beta_t\langle y^t - y^{t+1}, A(\bar{x}^{t+1} - x^{t+1}) \rangle \\
\geq& \phi_t^* + \frac{\beta_t}{2\eta_t}\left(\|x^t - \bar{x}^{t+1}\|^2_2 + \|\bar{x}^{t+1} - x^{t+1}\|^2_2\right)
    + \beta_t\left(g(\bar{x}^{t+1}) + \langle y^{t+1}, A\bar{x}^{t+1}\rangle \right) \\
&   - \beta_t\left(\frac{R^2\eta_t}{2}\|y^{t+1} - y^t\|^2_2
    + \frac{1}{2R^2\eta_t}\|A(\bar{x}^{t+1} - x^{t+1})\|^2_2 \right) \\
\geq& \phi_t^* + \frac{\beta_t}{2\eta_t}\left(\|x^t - \bar{x}^{t+1}\|^2_2 + \|\bar{x}^{t+1} - x^{t+1}\|^2_2\right)
    + \beta_t\left(g(\bar{x}^{t+1}) + \langle y^{t+1}, A\bar{x}^{t+1}\rangle \right) \\
&   - \beta_t\left(\frac{R^2\eta_t}{2}\|y^{t+1} - y^t\|^2_2
    + \frac{1}{2\eta_t}\|\bar{x}^{t+1} - x^{t+1}\|^2_2 \right) \\
\geq& \phi_t^*  - \frac{\beta_tR^2\eta_t}{2}\|y^{t+1} - y^t\|^2_2
+ \beta_t\left(g(\bar{x}^{t+1}) + \langle y^{t+1}, A\bar{x}^{t+1}\rangle \right),
\end{align*}
where the second inequality is due to Young's inequality and the third inequality is from \eqref{eq:norm_A}. This completes the proof.
\end{proof}

The next lemma concerns the update of the dual variable.
\begin{lemma}\label{lemma:y}
For any $y\in\R^n$, it holds that
\begin{align*}
&\frac{1}{2\tau_t}\left(\|y^t - y\|^2_2 - (1+\gamma\tau_t)\|y^{t+1} - y\|^2_2 - \|y^{t+1} - y^t\|^2_2\right) \\
\geq& \langle A\bar{x}^{t+1}, y - y^{t+1}\rangle + f^*(y^{t+1}) - f^*(y).
\numberthis\label{eq:lemma_y}
\end{align*}
\end{lemma}

\begin{proof}
Using \eqref{eq:alg1_y} and following similar derivation as \eqref{eq:derive_x_int}, it is easy to show that there exists $s\in\partial f^*(y^{t+1})$ such that the following holds:
\begin{align*}
&\|y^t - y\|^2_2 - \|y^{t+1} - y\|^2_2 - \|y^t - y^{t+1}\|^2_2 \\
=&2\langle y^t-y^{t+1}, y^{t+1}-y \rangle = 2\langle \tau_t(-A\bar{x}^{t+1} + s), y^{t+1}-y \rangle \\ 
\geq& 2\tau_t\left(\langle A\bar{x}^{t+1}, y - y^{t+1}\rangle + f^*(y^{t+1}) - f^*(y) + \frac{\gamma}{2}\|y^{t+1} -y\|^2_2 \right),
\numberthis \label{eq:lemma_y_1}
\end{align*}
where the inequality is due to the $\gamma$-strong convexity of $f^*(y)$, which is implied by the $(1/\gamma)$-smoothness of $f$ \cite{kakade2009duality}. Dividing \eqref{eq:lemma_y_1} by $2\tau_t$ yields \eqref{eq:lemma_y}.
\end{proof}

We are now ready to present the main convergence results of DAPD.
\begin{theorem}
\label{thm:alg1}
Consider the first $T$ iterations of DAPD. Assume the parameters satisfy \eqref{eq:condition_beta} and the following conditions:
\begin{gather}
\eta_t\tau_t\leq \frac{1}{R^2}, \label{eq:c2}\\
\frac{\beta_{t+1}}{\tau_{t+1}}\leq \frac{\beta_t}{\tau_t}(1+\gamma\tau_t). \label{eq:c3}
\end{gather}
Define
\begin{equation}
\hat{x}^T=\frac{1}{B_{t-1}}\sum_{t=0}^{T-1}\beta_t\bar{x}^{t+1}
\quad\text{and}\quad
\hat{y}^T=\frac{1}{B_{t-1}}\sum_{t=0}^{T-1}\beta_t y^{t+1}.
\end{equation}
The following inequality holds for any $x\in\R^d$ and $y\in\R^n$:
\begin{equation}\label{eq:thm_alg1_convergence}
F(\hat{x}^T,y) - F(x,\hat{y}^T) \leq \frac{1}{B_{T-1}}\left(\frac{\beta_0}{2\tau_0}\|y^0 - y\|^2_2 + \frac{1}{2}\|x^0-x\|^2_2\right).
\end{equation}
\end{theorem}

\begin{proof}
Multiplying \eqref{eq:lemma_y} by $\beta_t$, and adding the resulted inequality to \eqref{eq:change_phi}, we obtain
\begin{align*}
&\phi_{t+1}^*  - \phi_t^* + \frac{\beta_t}{2\tau_t}\left(\|y^t - y\|^2_2 - (1+\gamma\tau_t)\|y^{t+1} - y\|^2_2\right) \\
\geq& \beta_t\left(g(\bar{x}^{t+1}) + \langle y^{t+1}, A\bar{x}^{t+1}\rangle \right) - \frac{\beta_t R^2\eta_t}{2}\|y^{t+1} - y^t\|^2_2 \\
&   + \beta_t\left(\langle y-y^{t+1},A\bar{x}^{t+1} \rangle + f^*(y^{t+1}) - f^*(y) \right)+ \frac{\beta_t}{2\tau_t}\|y^{t+1} - y^t\|^2_2 \\
\geq& \beta_t\left(\langle y,A\bar{x}^{t+1} \rangle + g(\bar{x}^{t+1}) + f^*(y^{t+1}) - f^*(y)\right), \numberthis \label{eq:c4}
\end{align*}
where the last inequality is due to \eqref{eq:c2}. Combining \eqref{eq:c3} and \eqref{eq:c4} yields
\begin{align*}
&\left(\frac{\beta_t}{2\tau_t}\|y^t - y\|^2_2 - \phi_t^*\right)- \left(\frac{\beta_{t+1}}{2\tau_{t+1}}\|y^{t+1} - y\|^2_2 - \phi_{t+1}^*\right) \\
\geq& \beta_t\left(\langle y,A\bar{x}^{t+1} \rangle + g(\bar{x}^{t+1}) + f^*(y^{t+1}) - f^*(y)\right). \numberthis \label{eq:c5}
\end{align*}
Note that the left hand side of \eqref{eq:c5} has a telescoping structure. Summing \eqref{eq:c5} over $t=0,\ldots,T-1$ yields
\begin{align*}
&\sum_{t=0}^{T-1}\beta_t\left(\langle y,A\bar{x}^{t+1} \rangle + g(\bar{x}^{t+1}) + f^*(y^{t+1}) - f^*(y)\right) \\
\leq& \left(\frac{\beta_0}{2\tau_0}\|y^0 - y\|^2_2 - \phi_0^*\right) - \left(\frac{\beta_T}{2\tau_T}\|y^T - y\|^2_2 - \phi_T^*\right) \leq \frac{\beta_0}{2\tau_0}\|y^0 - y\|^2_2 - \phi_0^* + \phi_T^*.
\numberthis\label{eq:with_phi}
\end{align*}
From \eqref{eq:def_phi}, it is straightforward that $\phi_0^* = \min_{x\in\R^d} \frac{1}{2}\|x - x^0\|^2_2 = 0$ and
\[
\phi_T^* \leq \phi_T(x)
=\frac{1}{2}\|x-x^0\|^2_2 + \sum_{t=0}^{T-1}\beta_t\left(g(x) + \langle y^{t+1}, Ax \rangle \right).
\]
Combining these facts with \eqref{eq:with_phi} and using the convexity-concavity of $F(x,y)$, we have
\begin{align*}
&\frac{\beta_0}{2\tau_0}\|y^0 - y\|^2_2 + \frac{1}{2}\|x^0-x\|^2_2 \\
\geq&\sum_{t=0}^{T-1}\beta_t\left(\langle y,A\bar{x}^{t+1} \rangle - f^*(y) + g(\bar{x}^{t+1})
    -\langle y^{t+1},Ax \rangle + f^*(y^{t+1}) - g(x)\right) \\
=& \sum_{t=0}^{T-1}\beta_t\left(F(\bar{x}^{t+1},y) - F(x,y^{t+1})\right) \geq \left(\sum_{t=0}^{T-1}\beta_t\right)\cdot\left(F(\hat{x}^T,y) - F(x,\hat{y}^T)\right) \\
=& B_{T-1}\left(F(\hat{x}^T,y) - F(x,\hat{y}^T)\right),
\end{align*}
which completes the proof.
\end{proof}


From Theorem \ref{thm:alg1}, we can derive some more interpretable complexity bounds by choosing some specific parameters.
\begin{corollary}
\label{corollary:1}
The following facts hold for DAPD (Algorithm \ref{alg:1}).
\begin{enumerate}[label=(\roman*)]
\item If $\gamma>0$ and $\mu>0$, by choosing
    \begin{equation}
    \label{eq:case_1_par}
    \eta_t = \frac{1}{R}\sqrt{\frac{\gamma}{\mu}},\;\;
    \tau_t = \frac{1}{R}\sqrt{\frac{\mu}{\gamma}}
    \;\;\text{and}\;\;
    \beta_t = \frac{1}{R}\sqrt{\frac{\gamma}{\mu}}\left(1+\frac{\sqrt{\mu\gamma}}{R}\right)^t,
    \end{equation}
    DAPD converges linearly:
    \begin{align*}
    \|\hat{x}^T - x^*\|^2_2 \leq \frac{1}{\left(1+\frac{\sqrt{\mu\gamma}}{R}\right)^T - 1}\left[\|x^0 - x^*\|^2_2 + \frac{\gamma}{\mu}\|y^0 - y^*\|_2^2\right]. \numberthis \label{eq:case_1_conclude}
    \end{align*}
\item If $\gamma>0$, $\mu=0$ and $f$ is $L$-Lipschitz continuous, by choosing
    \[
        \eta_t=\beta_t=\frac{\gamma(t+1)}{3R^2}
        \;\;\text{and}\;\;
        \tau_t=\frac{3}{\gamma(t+1)},
    \]
    DAPD converges sublinearly in terms of primal sub-optimality:
    \begin{equation}\label{eq:case_2_conclude}
    P(\hat{x}^T) - P(x^*) \leq \frac{9R^2\sdist{x^0,X^*} + 4\gamma^2 L^2}{3\gamma T(T+1)}.
    \end{equation}
\item If $\mu>0$ and $\gamma=0$, by choosing
    \[
        \eta_t = \frac{4}{\mu(t+1)},\;\;
        \tau_t = \frac{\mu(t+1)}{4R^2}
        \;\;\text{and}\;\;
        \beta_t = \frac{2(t+1)}{\mu},
    \]
    DAPD converges sublinearly:
    \[
        \|\hat{x}^T - x^*\|^2_2 \leq \frac{\mu\|x^0 - x^*\|^2_2 + 8R^2\sdist{y^0,Y^*}}{\mu T(T+1)}.
    \]
\item If $\gamma=0$, $\mu=0$ and $f$ is $L$-Lipschitz continuous, by setting
    \[
        \tau_t \equiv \tau
        \;\;\text{and}\;\;
        \eta_t=\beta_t \equiv \frac{1}{\tau R^2},
    \]
    where $\tau>0$ is an arbitrary constant, we have
    \[
        P(\hat{x}^T) - P(x^*) \leq \frac{\tau R^2\cdot\sdist{x^0, X^*} + \frac{4L^2}{\tau}}{2T}.
    \]
\end{enumerate}
\end{corollary}

\begin{proof}
For the sake of succinctness, we only prove the first two cases, while the other two cases can be proved similarly.

\emph{Case (i): $\gamma>0$ and $\mu>0$.} 
It is easy to verify that the parameter setting in \eqref{eq:case_1_par} satisfies \eqref{eq:condition_beta}, \eqref{eq:c2} and \eqref{eq:c3}. Thus, Theorem \ref{thm:alg1} applies here. Choosing $(x,y)=(x^*,y^*)$ in \eqref{eq:thm_alg1_convergence} gives
\begin{equation}\label{cor-proof-eq-1}
F(\hat{x}^T,y^*) - F(x^*,\hat{y}^T)
\leq \frac{1}{B_{T-1}}\left(\frac{\beta_0}{2\tau_0}\|y^0 - y^*\|^2_2 + \frac{1}{2}\|x^0-x^*\|^2_2\right).
\end{equation}
The $\mu$-strong convexity of $F(\cdot,y^*)$ implies
\begin{align*}
F(\hat{x}^T,y^*) - F(x^*,\hat{y}^T) \geq F(\hat{x}^T,y^*) - F(x^*,y^*) \geq \frac{\mu}{2}\|\hat{x}^T - x^*\|^2_2. \numberthis \label{cor-proof-eq-2}
\end{align*}
Combining \eqref{cor-proof-eq-1}, \eqref{cor-proof-eq-2} and \eqref{eq:case_1_par} yields \eqref{eq:case_1_conclude}.

\emph{Case (ii): $\gamma>0$, $\mu=0$ and $f$ is Lipschitz continuous.} It is again easy to verify that the conditions in Theorem \ref{thm:alg1} are satisfied and thus Theorem \ref{thm:alg1} applies here. In \eqref{eq:thm_alg1_convergence}, we set $x=x^*$ and take supremum with respect to $y$ in the domain of $f^*$, which gives
\begin{align*}
\frac{1}{B_{T-1}}\left(\frac{\beta_0}{2\tau_0}\sup_{y\in\mathrm{dom}\, f^*}\|y^0 - y\|^2_2 + \frac{1}{2}\|x^0-x^*\|^2_2\right)
\geq& \sup_{y\in\mathrm{dom}\, f^* }F(\hat{x}^T,y) - F(x^*,\hat{y}^T) \\
\geq& P(\hat{x}^T) - P(x^*). \numberthis \label{cor-proof-case-2}
\end{align*}
Because $f$ is $L$-Lipschitz continuous, the domain of $f^*$ is bounded such that $\|y\|_2\leq L$ for all $y\in\mathrm{dom}\, f^*$ \cite{tan2018stochastic}. Hence, \eqref{cor-proof-case-2} implies
\begin{equation}\label{cor-proof-case-2-eq-2}
\frac{1}{B_{T-1}}\left(\frac{2\beta_0}{\tau_0}L^2 + \frac{1}{2}\|x^0-x^*\|^2_2\right)\geq P(\hat{x}^T) - P(x^*).
\end{equation}
Since \eqref{cor-proof-case-2-eq-2} holds for any $x^*\in X^*$, by replacing $\|x^0 - x^*\|^2_2$ by $\sdist{x^0, X^*}$ in \eqref{cor-proof-case-2-eq-2} we obtain the desired result \eqref{eq:case_2_conclude}. 
\end{proof}

\begin{remark}
For problem \eqref{eq:problem}, if $f$ is $(1/\gamma)$-smooth and $g$ is $\mu$-strongly convex, the condition number of problem \eqref{eq:problem} is $\kappa\coloneqq \frac{R^2}{\mu\gamma}$. The case (i) in Corollary \ref{corollary:1} implies that DAPD requires $\O\left(\sqrt{\kappa}\log\frac{1}{\epsilon}\right)$ iterations to achieve $\epsilon$ accuracy, which is an accelerated rate and matches the complexity lower bound of first-order methods.

On the other hand, when the objective function of \eqref{eq:problem} is smooth but non-strongly convex (case (ii)), or is non-smooth but strongly convex (case (iii)), Corollary \ref{corollary:1} implies that DAPD has $\O\left(\frac{1}{T^2}\right)$ accelerated convergence rate, which is also optimal for first-order methods. For non-smooth and non-strongly convex problems (case (iv)), the convergence rate of DAPD is $\O\left(\frac{1}{T}\right)$, which is faster than subgradient method and the original dual averaging method, whose rates are $\O(1/\sqrt{T})$ under the same assumptions. 

The assumption that $f$ is Lipschitz continuous required in cases (ii) and (iv) of Corollary \ref{corollary:1} is standard for primal-dual methods.
\end{remark}

{
\begin{remark}
Though our theoretical analysis is based on the averaged iterates $(\hat{x}^T, \hat{y}^T)$, in the actual implementation of our algorithms, we will always choose the last iterate $(x^T,y^T)$ as the output to make sure the solution structure (e.g., sparsity) will be preserved. Such strategy is also the common practice of dual-averaging-type methods \cite{xiao2010dual}.
\end{remark}
}


\section{The Stochastic DAPD Method}
\label{sec:stochastic}

In this section, we focus on \eqref{eq:problem} where $f$ has a finite-sum structure. More specifically, we assume that the primal problem is of the following form:
\begin{equation} \label{eq:stoc_primal}
\min_{x\in\R^d} \left\{\tilde{P}(x) \coloneqq \frac{1}{n}\sum_{i=1}^n f_i(a_i^\top x) + g(x)\right\},
\end{equation}
with $f_i:\R\rightarrow \R$. Problem \eqref{eq:stoc_primal} reduces to \eqref{eq:problem} by choosing $f(u)=\frac{1}{n}\sum_{i=1}^n f_i(u)$. The primal-dual formulation of \eqref{eq:stoc_primal} is:
\begin{equation*}
    \min_{x\in\R^d}\max_{y\in\R^n} \left\{\tilde{F}(x,y) \coloneqq \frac{1}{n}\ip{y,Ax} + g(x) - \frac{1}{n}\sum_{i=1}^n f_i^*(y_i)\right\}.
\end{equation*}

Since \eqref{eq:stoc_primal} is a special case of \eqref{eq:problem}, DAPD can be directly applied here. If we assume each $f_i$ is $(1/\gamma)$-smooth and $g$ is $\mu$-strongly convex, the complexity of DAPD for solving \eqref{eq:stoc_primal} is
$\O\left(\sqrt{\kappa'}\log\frac{1}{\epsilon}\right)$,
and $\kappa'\coloneqq \frac{R^2}{n\mu\gamma}$ denotes the condition number.
In this section, we show that by utilizing the finite-sum structure of $f$ in problem \eqref{eq:stoc_primal}, we can design a stochastic version of DAPD, which has a better complexity.

\begin{algorithm}[t]
    \caption{Stochastic Dual-Averaging Primal-Dual (SDAPD) Method}\label{alg:2}
    \begin{algorithmic}[1]
        \Parameter initial values $x^0$ and $y^0$, primal step sizes $\{\beta_t\}$ and $\eta$, dual step size $\tau$
        \State Initialize $\bar{x}^0=x^0$ and $B_0=\beta_0$
        \For{$t=0,1,\dots$}
            \State Uniformly randomly sample $i_t\in \{1,2,\ldots,n\}$
            \State Compute intermediate variable:
                \begin{equation}
                    \label{eq:alg2_x_int}
                    \bar{x}^{t+1} = \prox{\eta g}{x^t - \frac{\eta}{n} A^\top y^t}
                \end{equation}
            \State Update dual variable:
                \begin{equation}
                \label{eq:alg2_y}
                    y^{t+1}_i = \left\{ \begin{array}{ll}
                        \tilde{y}^{t+1}_i\coloneqq\prox{\tau f_i^*}{y^t_i + \tau \langle a_i,\bar{x}^{t+1}\rangle},\quad&\text{if }i=i_t \\
                        y^t_i, \quad&\text{if }i\neq i_t
                    \end{array}\right.
                \end{equation}
            \State Set
                \begin{equation}
                    \label{eq:alg2_y_bar}
                    \bar{y}^{t+1}=y^t+n(y^{t+1} - y^t)
                \end{equation}
            \State Update primal variable:
                \begin{equation}
                    \label{eq:alg2_x}
                    x^{t+1} = \prox{B_t g}{x^0 - s^{t+1}}, \mbox{ with } s^{t+1}\coloneqq\sum_{k=0}^t\frac{\beta_k}{n} A^\top \bar{y}^{k+1}
                \end{equation}
            \State Let $B_{t+1}\coloneqq B_t + \beta_{t+1}$
        \EndFor
    \end{algorithmic}
\end{algorithm}

Our stochastic method SDAPD, which is inspired by the stochastic primal-dual coordinate (SPDC) method \cite{zhang2017stochastic}, is presented in Algorithm \ref{alg:2}. 
In each iteration of SDAPD, only one coordinate of the dual variable $y_{i_t}$ is updated, with $i_t$ sampled uniformly random from $\{1,2,\ldots,n\}$.
Correspondingly, only one row vector $a_{i_t}^\top$ is involved in the update of the dual variable. Besides, another variable $\bar{y}^{t+1}$ is obtained by extrapolation. 
Moreover, note that in Algorithm \ref{alg:2} we only consider fixed primal and dual step sizes $\eta$ and $\tau$.

When implementing SDAPD, one should keep an auxiliary variable
\begin{equation}\label{eq:def_ut}
u^t\coloneqq \frac{1}{n}A^\top y^t.
\end{equation}
Since each time only one coordinate of $y$ is changed, $u_t$ can be updated incrementally as:
\begin{equation}\label{eq:update_ut}
u^{t+1}= u^t + \frac{1}{n}(y^{t+1}_{i_t} - y^t_{i_t})a_{i_t}.
\end{equation}
As a result, the matrix-vector multiplication in \eqref{eq:alg2_x} can be efficiently computed by:
\begin{align*}
\frac{1}{n}A^\top\bar{y}^{t+1} = \frac{1}{n}A^\top y^{t+1} + (y^{t+1}_{i_t} - y^t_{i_t})a_{i_t}=u^{t+1} + (y^{t+1}_{i_t} - y^t_{i_t})a_{i_t}.
\end{align*}
Therefore, the summation of gradients $s^{t+1}$ in \eqref{eq:alg2_x}
can also be incrementally updated with $\O(d)$ computation cost. As a result, the per-iteration complexity of SDAPD is $\O(d)$, much cheaper than
the per-iteration complexity $\O(nd)$ of DAPD.

\begin{remark}
We need to point out that Murata and Suzuki also developed an accelerated stochastic dual averaging method \cite{murata2017doubly} 
which is based on stochastic variance-reduction techniques \cite{johnson2013accelerating} and requires the assumption that $f_i$ is smooth.
\end{remark}


We now provide the convergence analysis of SDAPD. Here we make the following assumption.
\begin{assumption}\label{assumption:stoc}
All $f_i$'s are $(1/\gamma)$-smooth ($\gamma>0$), and $g(x)$ is $\mu$-strongly convex ($\mu>0$).
\end{assumption}

For the ease of presentation, we denote $f^*(y)\coloneqq \frac{1}{n}\sum_{i=1}^n f_i^*(y_i)$ throughout this section. Besides, we use $\F_t$ to stand for the $\sigma$-field generated by all random variables up to iteration $t$. Clearly, when conditioned on $\F_t$, $x^t$ and $y^t$ are known.
Similar to the analysis of DAPD, we define a potential function as follows:
\begin{equation}\label{def-phi'}
\tilde{\phi}_t(x)\coloneqq \frac{1}{2}\left\|x-x^0\right\|^2_2
    + \sum_{k=0}^{t-1} \beta_k\left(g(x) + \frac{1}{n}\langle \bar{y}^{k+1}, Ax\rangle \right).
\end{equation}
Again, it is easy to see that $x^{t+1}$ is the minimizer of $\tilde{\phi}_{t+1}(x)$. Since the updates of $\bar{x}^{t+1}$ and $x^{t+1}$ in SDAPD are almost identical to DAPD, we have the following lemma that is similar to Lemma \ref{lemma:x}.
\begin{lemma}\label{lemma:stoc_x}
Assume $\eta(1+B_{t-1}\mu)\geq\beta_t$. We have
\begin{align*}
\E[\F_t]{\tilde{\phi}^*_{t+1} - \tilde{\phi}^*_t}
\geq \beta_t\E[\F_t]{g(\bar{x}^{t+1}) + \frac{1}{n}\langle \bar{y}^{t+1}, A\bar{x}^{t+1}\rangle }
    - \frac{\bar{R}^2\beta_t\eta}{2}\E[\F_t]{\|y^{t+1} - y^t\|^2_2}.
\numberthis\label{eq:lemma_x_2}
\end{align*}
\end{lemma}

\begin{proof}
The proof is largely the same as Lemma \ref{lemma:x}.
Following the same argument as in Lemma \ref{lemma:x}, it is easy to show that \eqref{eq:x_phi_star} becomes
\begin{align*}
 \tilde{\phi}_{t+1}^*
\geq \tilde{\phi}_t^* + \frac{\beta_t}{2\eta_t}\|x^t - x^{t+1}\|^2_2
    + \beta_t\left(g(x^{t+1}) + \frac{1}{n}\langle \bar{y}^{t+1}, Ax^{t+1}\rangle \right).
\numberthis \label{eq:x_phi_star-new}
\end{align*}
and \eqref{eq:derive_x_int} becomes
\begin{align*}
&\|x^t - x^{t+1}\|^2_2 - \|x^t - \bar{x}^{t+1}\|^2_2 - \|\bar{x}^{t+1} - x^{t+1}\|^2_2 \\
\geq & 2\eta_t\left(\frac{1}{n}\langle y^t, A(\bar{x}^{t+1} - x^{t+1}) \rangle
    +  g(\bar{x}^{t+1}) - g(x^{t+1})\right).
\numberthis \label{eq:derive_x_int-new}
\end{align*}
Combining \eqref{eq:x_phi_star-new} and \eqref{eq:derive_x_int-new} yields
\begin{align*}
& \tilde{\phi}^*_{t+1} \\
\geq& \tilde{\phi}^*_t + \frac{\beta_t}{2\eta_t}\left(\|x^t - \bar{x}^{t+1}\|^2_2 + \|\bar{x}^{t+1} - x^{t+1}\|^2_2\right)
    + \beta_t\left(g(\bar{x}^{t+1}) + \frac{1}{n}\langle \bar{y}^{t+1}, A\bar{x}^{t+1}\rangle \right) \\
&   + \frac{\beta_t}{n}\langle y^t - \bar{y}^{t+1}, A(\bar{x}^{t+1} - x^{t+1}) \rangle \\
\geq& \tilde{\phi}^*_t + \frac{\beta_t}{2\eta_t}\left(\|x^t - \bar{x}^{t+1}\|^2_2 + \|\bar{x}^{t+1} - x^{t+1}\|^2_2\right)
    + \beta_t\left(g(\bar{x}^{t+1}) + \frac{1}{n}\langle \bar{y}^{t+1}, A\bar{x}^{t+1}\rangle \right) \\
&   - \frac{\beta_t}{n}\left( \frac{\eta_t}{2n}\|A^\top(y^t - \bar{y}^{t+1})\|_2^2 + \frac{n}{2\eta_t}\|\bar{x}^{t+1} - x^{t+1}\|_2^2) \right).
\numberthis\label{eq:alg2_x_1}
\end{align*}
where the last inequality is due to Young's inequality.
By noting that $y^t$ and $\bar{y}^{t+1}$ only differ in coordinate $i_t$, we have
\begin{align*}
\|A^\top(y^t - \bar{y}^{t+1})\|_2^2 = \|(y^t_{i_t} - \bar{y}^{t+1}_{i_t})a_{i_t}\|_2^2
\leq (y^t_{i_t} - \bar{y}^{t+1}_{i_t})^2\bar{R}^2
= \bar{R}^2\|y^t -\bar{y}^{t+1}\|^2_2,
\end{align*}
which combining with \eqref{eq:alg2_x_1} yields
\begin{align*}
& \tilde{\phi}^*_{t+1} \\
\geq& \tilde{\phi}^*_t + \frac{\beta_t}{2\eta_t}\left(\|x^t - \bar{x}^{t+1}\|^2_2 + \|\bar{x}^{t+1} - x^{t+1}\|^2_2\right)
    + \beta_t\left(g(\bar{x}^{t+1}) + \frac{1}{n}\langle \bar{y}^{t+1}, A\bar{x}^{t+1}\rangle \right) \\
&   - \frac{\beta_t}{n}\left( \frac{\bar{R}^2\eta_t}{2n}\|y^t - \bar{y}^{t+1}\|_2^2 + \frac{n}{2\eta_t}\|\bar{x}^{t+1} - x^{t+1}\|_2^2) \right) \\
\geq& \tilde{\phi}^*_t
    +\beta_t\left(g(\bar{x}^{t+1}) + \frac{1}{n}\langle \bar{y}^{t+1}, A\bar{x}^{t+1}\rangle \right)
   - \frac{\bar{R}^2\eta_t\beta_t}{2n^2}\|y^t - \bar{y}^{t+1}\|_2^2. \numberthis \label{eq:lemma-3.2-1}
\end{align*}
Using \eqref{eq:alg2_y_bar} and taking conditional expectation to \eqref{eq:lemma-3.2-1} yields the desired result \eqref{eq:lemma_x_2}.
\end{proof}


Similarly, we have the following lemma that is analogous to Lemma \ref{lemma:y}. We omit the proof for succinctness.
\begin{lemma}
For each $i\in\{1,2,\dots,n\}$, it holds that
\begin{align*}
&\frac{1}{2\tau}\left[(y^t_i - y_i)^2 - \left(1+\gamma\tau\right)(\tilde{y}^{t+1}_i - y_i)^2
    - (\tilde{y}^{t+1}_i - y^t_i)^2\right] \\
\geq& \langle (y_i - \tilde{y}^{t+1}_i)a_i,\bar{x}^{t+1}\rangle + f^*_i(\tilde{y}^{t+1}_i) - f^*_i(y_i), \qquad \forall y_i\in\R.
\numberthis\label{eq:y_single}
\end{align*}
\end{lemma}
Moreover, we have the following lemma.
\begin{lemma}
When conditioning on $\F_t$, for any $y\in\R^n$, it holds that
\begin{align*}
&\frac{1}{2\tau}\E[\F_t]{\left(1+\frac{(n-1)\gamma\tau}{n}\right)\|y^t - y\|_2^2
    - \left(1+\gamma\tau\right)\|y^{t+1} - y\|_2^2
    - \|y^{t+1} - y^t\|_2^2} \\
\geq & \E[\F_t]{-\frac{1}{n}\langle \bar{y}^{t+1} - y, A\bar{x}^{t+1} \rangle
    + n f^*(y^{t+1}) - (n-1)f^*(y^t) - f^*(y)}.
\numberthis\label{eq:lemma_y_2}
\end{align*}
\end{lemma}
\begin{proof}
Note that when conditioning on $\F_t$, $\bar{x}^{t+1}$ is deterministic and independent of $i_t$. Hence, for each $i$, $y^{t+1}_i=\tilde{y}^{t+1}_i$ with probability $1/n$, and $y^{t+1}_i=y^t_i$ with probability $(n-1)/n$. This implies the following relationships that hold for any $y\in\R^n$:
\begin{align*}
\E[\F_t]{(y^{t+1}_i - y_i)^2} =& \frac{1}{n}(\tilde{y}^{t+1}_i - y_i)^2 + \frac{n-1}{n}(y^t_i - y_i)^2,\\
\E[\F_t]{(y^{t+1}_i - y^t_i)^2} =& \frac{1}{n}(\tilde{y}^{t+1}_i - y_i^t)^2,\\
\E[\F_t]{y^{t+1}_i} =& \frac{1}{n}\tilde{y}^{t+1}_i + \frac{n-1}{n}y^t_i, \\
\E[\F_t]{f_i^*(y^{t+1}_i)} =& \frac{1}{n}f_i^*(\tilde{y}^{t+1}_i) + \frac{n-1}{n}f_i^*(y^t_i).
\end{align*}
Plugging these relationships into \eqref{eq:y_single}, we obtain:
\begin{align*}
&\frac{1}{2\tau}\E[\F_t]{(n+(n-1)\gamma\tau)(y^t_i - y_i)^2
    - n\left(1+\gamma\tau\right)(y^{t+1}_i - y_i)^2
    - n(y^{t+1}_i - y^t_i)^2} \\
\geq& \left\langle \left(y_i - n\E[\F_t]{y^{t+1}_i} + (n-1)y^t_i\right)a_i,\bar{x}^{t+1}\right\rangle
    + n \E[\F_t]{f^*_i(y^{t+1}_i)} - (n-1)f^*_i(y^t_i) - f^*_i(y_i).
\end{align*}
Summing this inequality for $i\in\{1,2,\dots,n\}$ and using \eqref{eq:alg2_y_bar}, we get:
\begin{align*}
&\frac{1}{2\tau}\E[\F_t]{\left(1+\frac{(n-1)\gamma\tau}{n}\right)\|y^t - y\|_2^2-\left(1+\gamma\tau\right)\|y^{t+1}-y\|_2^2- \|y^{t+1} - y^t\|_2^2} \\
\geq& \frac{1}{n}\langle y- n\E[\F_t]{y^{t+1}}+(n-1)y^t, A\bar{x}^{t+1} \rangle + n \E[\F_t]{f^*(y^{t+1})} - (n-1)f^*(y^t) - f^*(y) \\
=& \frac{1}{n}\E[\F_t]{-\langle \bar{y}^{t+1}-y, A\bar{x}^{t+1} \rangle} + n \E[\F_t]{f^*(y^{t+1})} - (n-1)f^*(y^t) - f^*(y),
\end{align*}
which is the desired inequality \eqref{eq:lemma_y_2}.
\end{proof}

Now, we are ready to provide the convergence complexity for SDAPD (Algorithm \ref{alg:2}).
\begin{theorem}\label{thm:alg2}
Assume Assumption \ref{assumption:stoc} holds. We choose algorithm parameters as
\begin{gather*}
\eta=\frac{1}{\bar{R}}\sqrt{\frac{\gamma}{n\mu}},\quad
\tau=\frac{1}{\bar{R}}\sqrt{\frac{n\mu}{\gamma}},\quad
\beta_t=\frac{1}{\bar{R}}\sqrt{\frac{\gamma}{n\mu}}\cdot\xi^t, \mbox{ with } \xi\coloneqq 1 + \frac{1}{n+\bar{R}\sqrt{n/(\mu\gamma)}}.
\end{gather*}
Consider the first $T$ iterations of SDAPD and define $\hat{x}^T=\frac{1}{B_{T-1}}\sum_{t=0}^{T-1}\beta_t\bar{x}^{t+1}$, SDAPD converges linearly in expectation:
\begin{align*}
\E{\|\hat{x}^T - x^*\|^2_2}\leq \frac{\Delta_0}{\xi^T - 1},
\end{align*}
where $\Delta_0$ is a constant depending on $\bar{R}$, the initial point $(x^0,y^0)$ and optimal solution $(x^*,y^*)$. {Note that $(x^*,y^*)$ is unique here due to the strong convexity-concavity assumption}.
\end{theorem}

\begin{proof}
When conditioning on $\F_t$, we multiply \eqref{eq:lemma_y_2} by $\beta_t$ and add it to \eqref{eq:lemma_x_2}. We have
\begin{align*}
& \E[\F_t]{\tilde{\phi}^*_{t+1} - \tilde{\phi}^*_t}+ \frac{\beta_t}{2\tau}\E[\F_t]{\left(1+\frac{(n-1)\gamma\tau}{n}\right)\|y^t - y\|_2^2
    - \left(1+\gamma\tau\right)\|y^{t+1} - y\|_2^2 } \\
\geq&  \beta_t\E[\F_t]{g(\bar{x}^{t+1}) + \frac{1}{n}\langle \bar{y}^{t+1}, A\bar{x}^{t+1}\rangle}- \frac{\bar{R}^2\beta_t\eta}{2}\E[\F_t]{\|y^{t+1} - y^t\|^2_2} \\
&+ \beta_t\E[\F_t]{-\frac{1}{n}\langle \bar{y}^{t+1}-y, A\bar{x}^{t+1} \rangle + n f^*(y^{t+1}) - (n-1)f^*(y^t) - f^*(y)}  \\
&+ \frac{\beta_t}{2\tau}\E[\F_t]{\|y^{t+1} -y^t\|_2^2} \\
=& \beta_t\E[\F_t]{g(\bar{x}^{t+1}) + \frac{1}{n}\langle y,A\bar{x}^{t+1} \rangle + n f^*(y^{t+1}) - (n-1)f^*(y^t) - f^*(y)},
\numberthis\label{eq:alg2_key}
\end{align*}
where the equality uses the fact $\eta\tau= 1/\bar{R}^2$. Note that our parameters satisfy $\beta_t(1+\gamma\tau)\geq \beta_{t+1}\alpha$, where $\alpha\coloneqq 1 + \frac{(n-1)\gamma\tau}{n}$, from which we can upper bound the left-hand-side of \eqref{eq:alg2_key} by
\begin{align*}
& \E[\F_t]{\tilde{\phi}^*_{t+1}} - \tilde{\phi}^*_t
    + \frac{\alpha\beta_t}{2\tau}\|y^t-y\|_2^2 - \frac{\alpha\beta_{t+1}}{2\tau}\E[\F_t]{\|y^{t+1}-y\|_2^2} \\
=& \left(\frac{\alpha\beta_t}{2\tau}\|y^t-y\|_2^2 - \tilde{\phi}^*_t\right)
    - \E[\F_t]{\frac{\alpha\beta_{t+1}}{2\tau}\|y^{t+1}-y\|_2^2 - \tilde{\phi}^*_{t+1}}.
\end{align*}
Therefore, \eqref{eq:alg2_key} reduces to:
\begin{align*}
& \left(\frac{\alpha\beta_t}{2\tau}\|y^t-y\|_2^2 - \tilde{\phi}^*_t\right)
    - \E[\F_t]{\frac{\alpha\beta_{t+1}}{2\tau}\|y^{t+1}-y\|_2^2 - \tilde{\phi}^*_{t+1}}. \\
\geq& \beta_t\E[\F_t]{g(\bar{x}^{t+1}) + \frac{1}{n}\langle y,A\bar{x}^{t+1} \rangle
    + n f^*(y^{t+1}) - (n-1)f^*(y^t) - f^*(y)}. \numberthis \label{thm-3.5-eq-1}
\end{align*}
Summing \eqref{thm-3.5-eq-1} over $t=0, \ldots, T-1$ and apply total expectation, we obtain:
\begin{align*}
& \left(\frac{\alpha\beta_0}{2\tau}\|y^0-y\|_2^2 - \tilde{\phi}^*_0\right)
    - \E{\frac{\alpha\beta_{T}}{2\tau}\|y^T-y\|_2^2 - \tilde{\phi}^*_T} \\
\geq& \sum_{t=0}^{T-1}\beta_t\E{g(\bar{x}^{t+1}) + \frac{1}{n}\langle y,A\bar{x}^{t+1} \rangle
    + n f^*(y^{t+1}) - (n-1)f^*(y^t) - f^*(y)}.
\numberthis\label{eq:alg2_key_2}
\end{align*}
Using \eqref{def-phi'} and \eqref{eq:alg2_y_bar}, it is easy to see that $\tilde{\phi}^*_0=0$ and
\begin{align*}
\tilde{\phi}^*_T
\leq & \frac{1}{2}\left\|x-x^0\right\|^2_2 + \sum_{t=0}^{T-1} \beta_t\left(g(x) + \frac{1}{n}\langle ny^{t+1}-(n-1)y^t, Ax\rangle \right), \quad \forall x\in\R^n.
\end{align*}
Plugging these to \eqref{eq:alg2_key_2} and dropping the term $\|y^T-y\|_2^2$, we obtain:
\begin{align*}\numberthis\label{eq:last}
& \frac{1}{2}\left\|x-x^0\right\|^2_2 + \frac{\alpha\beta_0}{2\tau}\|y^0-y\|_2^2 \\
\geq& \sum_{t=0}^{T-1}\beta_t\E{g(\bar{x}^{t+1}) - g(x) + \frac{1}{n}\langle y,A(\bar{x}^{t+1}-x) \rangle} \\
&   + \sum_{t=0}^{T-1}\beta_t\E{-\frac{1}{n}\langle ny^{t+1}-(n-1)y^t - y, Ax\rangle + n f^*(y^{t+1}) - (n-1)f^*(y^t) - f^*(y)} \\
=& \sum_{t=0}^{T-1}\beta_t\E{\tilde{F}(\bar{x}^{t+1},y) - \tilde{F}(x,y)} +\sum_{t=0}^{T-1}\beta_t\E{-n\tilde{F}(x,y^{t+1}) + (n-1)\tilde{F}(x,y^t) + \tilde{F}(x,y)}.
\end{align*}
Now, we choose $(x,y)=(x^*,y^*)$. The first term on the right-hand-side of \eqref{eq:last} can be bounded by:
\begin{align*}
\sum_{t=0}^{T-1}\beta_t\E{\tilde{F}(\bar{x}^{t+1},y^*) - \tilde{F}(x^*,y^*)}
\geq& B_{T-1}\E{\tilde{F}(\hat{x}^T,y^*) - \tilde{F}(x^*,y^*)} \\
\geq& \frac{B_{T-1}\mu}{2}\E{\|\hat{x}^T - x^*\|^2_2},
\numberthis\label{eq:last_t2}
\end{align*}
where the $\mu$-strong convexity of $F(\cdot,y^*)$ and the definition of $\hat{x}^T$ are used.
By using the fact $\tilde{F}(x^*,y^*)-\tilde{F}(x^*,y)\geq 0$ for any $y$, we can bound the second term on the right-hand-side of \eqref{eq:last} as:
\begin{align*}
& \sum_{t=0}^{T-1}\beta_t\E{-n\tilde{F}(x^*,y^{t+1}) + (n-1)\tilde{F}(x^*,y^t) + \tilde{F}(x^*,y^*)} \\
=& \sum_{t=0}^{T-1}\beta_t\E{n\left(\tilde{F}(x^*,y^*) - \tilde{F}(x^*,y^{t+1})\right) - (n-1)\left(\tilde{F}(x^*,y^*) - \tilde{F}(x^*,y^t)\right)} \\
=& \sum_{t=1}^{T-1}\left(n\beta_{t-1} - (n-1)\beta_t\right)\E{\tilde{F}(x^*,y^*) - \tilde{F}(x^*,y^t)} \\
&+n\beta_{T-1}\E{\tilde{F}(x^*,y^*) - \tilde{F}(x^*,y^T)} - (n-1)\beta_0\left(\tilde{F}(x^*,y^*) - \tilde{F}(x^*,y^0)\right) \\
\geq& - (n-1)\beta_0\left(\tilde{F}(x^*,y^*) - \tilde{F}(x^*,y^0)\right),
\numberthis\label{eq:last_t1}
\end{align*}
where the inequality follows from the fact that $n\beta_{t-1}\geq (n-1)\beta_t$.
Combining \eqref{eq:last}, \eqref{eq:last_t2} and \eqref{eq:last_t1} gives
\begin{align*}
 \frac{1}{2}\left\|x^0-x^*\right\|^2_2 + \frac{\alpha\beta_0}{2\tau}\|y^0-y^*\|_2^2
    + (n-1)\beta_0\left(\tilde{F}(x^*,y^*) - \tilde{F}(x^*,y^0)\right) \geq \frac{B_{T-1}\mu}{2}\E{\|\hat{x}^T - x^*\|^2_2},
\end{align*}
which leads to the desired result.
\end{proof}

\begin{remark}
Under Assumption \ref{assumption:stoc}, the condition number of problem \eqref{eq:stoc_primal}
usually defined in stochastic optimization literature (see, e.g., \cite{zhang2017stochastic}) is $\bar{\kappa}'\coloneqq \frac{\bar{R}^2}{\mu\gamma}$.
Note that $\kappa'\leq\bar{\kappa}'\leq n\kappa'$. Therefore, Theorem \ref{thm:alg2} implies that the number of iterations needed by SDAPD to achieve $\epsilon$-accuracy is
\begin{equation}
\label{eq:complexity_stochastic}
\O\left(\left(n+\sqrt{n\bar{\kappa}'}\right)\log\frac{1}{\epsilon}\right),
\end{equation}
which matches the lower bound of the complexity of stochastic first-order methods \cite{lan2017optimal}. Moreover, even though $\bar{\kappa}'$ might be larger than $\kappa'$ in DAPD, \eqref{eq:complexity_stochastic} still suggests that SDAPD is faster than DAPD, given that each iteration of DAPD is approximately $n$ times more expensive than SDAPD.

From these results, we can conclude that SDAPD is better than {regularized dual averaging,
the stochastic dual averaging method for minimizing the composite objective function},
whose complexity is in the order of $\O(1/\epsilon)$ under the same assumption \cite{xiao2010dual}. Besides, \eqref{eq:complexity_stochastic} also implies that SDAPD is better than some variance-reduced stochastic methods such as ProxSVRG \cite{xiao2014proximal}, whose complexity is
\[
\O\left(\left(n+\bar{\kappa}'\right)\log\frac{1}{\epsilon}\right),
\]
when the condition number $\bar{\kappa}'$ is larger than $n$. Though some accelerated stochastic methods like Katyusha \cite{allen2017katyusha} and SPDC \cite{zhang2017stochastic} have the same complexity as SDAPD, we will show later that SDAPD is more powerful when the data matrix $A$ is sparse.
\end{remark}

\begin{remark}
\textbf{Generalization to non-smooth or non-strongly-convex problems.}
Our results in this section can be extended to non-smooth or non-strongly convex problems easily, by slightly perturbing the primal-dual formulation.
When $f_i$ is non-smooth, we can augment $f_i^*$ as $\tilde{f}_i^*(y_i) \coloneqq f_i^*(y_i) + \frac{\delta_1}{2}(y_i)^2$. While $g$ is non-strongly convex, it can be perturbed as $\tilde{g}(x) \coloneqq g(x) + \frac{\delta_2}{2}\|x\|^2_2$. Here both $\delta_1$ and $\delta_2$ are small constants that are proportional to the desired solution accuracy $\epsilon$. Following such strategy, we can easily derive the complexities of SDAPD in different seniors, which are presented in Table \ref{table:complexity}. The derivation is similar to the one in \cite{zhang2017stochastic}, and we omit the details here for succinctness.
\end{remark}

\begin{table}
\centering
\caption{Iteration complexities of SDPAD for achieving $\epsilon$-solution accuracy under different settings.
Some constants and logarithmic factors are hidden.}
\begin{tabular}{|c|c|c|}
\hline
                    & {$g(x)$} $\mu$-strongly convex                     & {$g(x)$} non-strongly convex \\ \hline
{$f_i(u)$} $(1/\gamma)$-smooth & $\left(n+\bar{R}\sqrt{n/(\mu\gamma)}\right)\log (1/\epsilon)$ &  $n+\bar{R}\sqrt{n/(\mu\epsilon)}$ \\ \hline
{$f_i(u)$} non-smooth          & $n+\bar{R}\sqrt{n/(\gamma\epsilon)}$  & $n+\bar{R}\sqrt{n}/\epsilon$ \\ \hline
\end{tabular}
\label{table:complexity}
\end{table}

\section{Efficient Implementation of SDAPD on Sparse Data}\label{sec:sparse}

In this section, we focus on the case that each vector $a_i$ is a sparse vector so that the data matrix $A$ is also sparse. We show how to efficiently implement SDAPD (Algorithm \ref{alg:2}) on problems with sparse $A$, which can further reduce the per-iteration complexity of SDAPD from $\O(d)$
to $\O(\rho d)$. 
Throughout this section, we make the following assumption on function $g$.
\begin{assumption}\label{assumption:sparse}
Assume $g(x)$ is separable, i.e., it can be decomposed as $g(x)=\sum_{j=1}^d g_j(x_j)$.
\end{assumption}


{
Here we briefly explain why dual-averaging type algorithm can promote the sparsity. For ease of discussion, we denote $h(x)\coloneqq f(Ax)$. The dual averaging update is:
\begin{equation}\label{eq:da_example}
    z^{t+1}=\prox{B_t g}{z^0 - \sum_{k=0}^t \beta_k \nabla h(z^k)}.
\end{equation}
Note that there is no \emph{direct} dependence between any two consecutive iterates $z^t$ and $z^{t+1}$. The only place where $z^t$ influences $z^{t+1}$ is in estimating the gradient $\nabla h(z^t)$. In many problems with sparse data, the gradient function $\nabla h(z)$ also possesses sparse structure where only a small portion of coordinates of $z^t$ is required for evaluating $\nabla h(z^t)$. Therefore, dual averaging methods allow lazy sparse update, which only updates the coordinates that will be involved in evaluating the next gradient.
}

Other types existing stochastic algorithms are incapable of admitting sparse update, except on certain problems with special structures. (See more details in Remark \ref{rem:sparse-update}).
For example, the gradient might not be sparse for some methods like SAGA \cite{defazio2014saga}, even when the problem data is sparse. Moreover, some accelerated methods require an extrapolation step which requires to add two dense vectors. In the following, we show how to efficiently implement SDAPD for sparse data.

When implementing SDAPD, we need to keep two auxiliary variables $u^t$ and $s^t$, which are defined in \eqref{eq:def_ut} and \eqref{eq:alg2_x} respectively. With $u^t$ and $s^t$ on hand, any coordinate of $x^t$, say $x^t_j$, can be recovered via
\[
x^t_j = \prox{B_{t-1}g_j}{x^0_j - s^t_j}
\]
in only $\O(1)$ time, due to the separable assumption of $g(x)$. Similarly, the $j$-th coordinate of $\bar{x}^{t+1}$ can be computed by
\[
\bar{x}^{t+1}_j = \prox{\eta g_j}{x^t_j - \eta u^t_j}.
\]
Note that $x^t$ is only used in the update of $\bar{x}^{t+1}$, while the only role of $\bar{x}^{t+1}$ is for computing the inner product $a_{i_t}^\top \bar{x}^{t+1}$ in \eqref{eq:alg2_y}. This implies that we do not need to evaluate $x^t_j$ and $\bar{x}^{t+1}_j$ when $a_{i_t,j}=0$. Using this property, the whole iteration of SDAPD can be done in $\O(\|a_{i_t}\|_0)$ computational cost.

Now, the remaining problem is how to update $u^{t+1}$ and $s^{t+1}$ for sparse data. For $u^{t+1}$, it is straightforward by using \eqref{eq:update_ut}, which adds a sparse vector $a_{i_t}$ to $u^t$ in each iteration. The real challenge is how to update $s^{t+1}$ in \eqref{eq:alg2_x}, because it is a summation of dense vectors. Here, we present a novel way to sparsify the update of $s^{t+1}$, by decomposing it into the combination of two sequences. For the ease of discussion, we define
\begin{equation}\label{def:delta}
\delta^t \coloneqq \frac{(y_{i_t}^{t+1} - y_{i_t}^t)}{n}\cdot a_{i_t}.
\end{equation}
Hence, $\delta^t$ is a sparse vector if $a_{i_t}$ is sparse. We need to show the following lemma first.

\begin{lemma}
Consider SDAPD (Algorithm \ref{alg:2}) with $\beta_t$ chosen in the form of $\beta_t=\beta_0\theta^{-t}$ for some $\theta\in(0,1)$.
Define two sequences:
\[
v^{t+1}\coloneqq -\frac{\beta_0\theta}{n(1 - \theta)}A^\top y^0
    + \sum_{k=0}^t\beta_k\left(n - \frac{1}{1-\theta}\right)\delta^k, \mbox{ and }
w^{t+1}\coloneqq \frac{1}{n(1-\theta)}A^\top y^0
    + \frac{1}{1 - \theta}\sum_{k=0}^t\delta^k.
\]
It holds that
\begin{equation}\label{eq:equiv}
s^{t+1} \coloneqq \sum_{k=0}^t \frac{\beta_k}{n}A^\top \bar{y}^{k+1} = v^{t+1} + \beta_t w^{t+1}
\end{equation}
\end{lemma}

\begin{proof}
We prove \eqref{eq:equiv} by induction. we first note two useful relationships:
\begin{align}
\frac{1}{n}A^\top y^{t+1} =& \frac{1}{n}A^\top y^t + \delta^t   \label{eq:r1}\\
\frac{1}{n}A^\top \bar{y}^{t+1} =& \frac{1}{n}A^\top y^t + n\delta^t, \label{eq:r2}
\end{align}
which are easy to be obtained from \eqref{def:delta}, \eqref{eq:alg2_y} and \eqref{eq:alg2_y_bar}.

When $t=0$,
\[
v^{t+1}=v^1 = - \frac{\beta_0\theta}{n(1-\theta)}A^\top y^0 + \beta_0\left(n - \frac{1}{1-\theta}\right)\delta^0
\]
and
\[
\beta_t w^{t+1}=\beta_0 w^1 = \frac{\beta_0}{n(1-\theta)}A^\top y^0
    + \frac{\beta_0}{1 - \theta}\delta^0.
\]
By adding these two equations together, we have:
\begin{align*}
v^1+\beta_0 w^1 = \frac{\beta_0}{n}A^\top y^0 + n\beta_0\delta_0 = \frac{\beta_0}{n}A^\top \bar{y}^1,
\end{align*}
where the last equality follows from \eqref{eq:r2}. So \eqref{eq:equiv} is proved for $t=0$.

Now we assume that \eqref{eq:equiv} holds for $t-1$, i.e.,
\[
v^t + \beta_{t-1} w^t = \sum_{k=0}^{t-1} \frac{\beta_k}{n}A^\top \bar{y}^{k+1}.
\]
Thus
\begin{align*}
v^{t+1} + \beta_t w^{t+1} = \left(v^{t+1} - v^t\right) + \left(\beta_t w^{t+1} - \beta_{t-1} w^t\right)
   +\sum_{k=0}^{t-1} \frac{\beta_k}{n}A^\top \bar{y}^{k+1}.
\numberthis \label{eq:r3}
\end{align*}
From \eqref{eq:r1} and the fact $\beta_{t-1}=\beta_t\theta$, we have:
\begin{align*}
\beta_t w^{t+1} - \beta_{t-1} w^t
= \frac{\beta_t}{1-\theta}\delta^t + \frac{\beta_t}{n}A^\top y^t,
\end{align*}
Hence,
\begin{align*}
\left(v^{t+1} - v^t\right) + \left(\beta_t w^{t+1} - \beta_{t-1} w^t\right)
= \beta_t\left(n - \frac{1}{1-\theta}\right)\delta^t + \frac{\beta_t}{1-\theta}\delta^t + \frac{\beta_t}{n}A^\top y^t
= \frac{\beta_t}{n}A^\top \bar{y}^{t+1},
\end{align*}
which is due to \eqref{eq:r2} again. Combining this equation with \eqref{eq:r3} proves \eqref{eq:equiv}.
\end{proof}


\begin{remark}
Note that both $v^{t+1}$ and $w^{t+1}$ are actually the summation of sparse vectors $\delta^t$, except the first term $A^\top y^0$. As a result, after computing $A^\top y^0$ at the very beginning of the algorithm, both $v^{t+1}$ and $w^{t+1}$ can be updated in a sparse way. With the help of these two sequences, the whole algorithm is capable of doing sparse update,
and thus has only $\O(\rho d)$ per-iteration complexity on average instead of $\O(d)$.
In many large scale applications, $\rho$ can be very small like $\rho\approx 10^{-3}$
or even smaller. For example, the well-known DBLP dataset has the sparsity
$\rho\approx 2.0\times 10^{-5}$ \cite{yang2015defining}.
Hence, sparse update can bring great acceleration on such problems.
\end{remark}

We now continue the discussion on theoretical complexity of SDAPD. After combining the sparse update technique discussed above, the overall computation cost of SDAPD to achieve $\epsilon$-accuracy becomes
\[
\O\left(\rho d\left(n+\sqrt{n\bar{\kappa}'}\right)\log\frac{1}{\epsilon}\right)
\]
for strongly convex and smooth problems, if we take both convergence rate and per-iteration computation cost into consideration. This complexity is  better than the complexity of existing accelerated stochastic methods like SPDC, namely,
\[
\O\left(d\left(n+\sqrt{n\bar{\kappa}'}\right)\log\frac{1}{\epsilon}\right),
\]
due to the factor $\rho$ ($0<\rho\leq 1$).

{
\begin{remark}
Lee and Sidford proposed an efficient implementation of accelerated coordinate descent in \cite{lee2013efficient}, which shares similar idea of decomposing the updates into two sequences that can be updated efficiently. However, the motivation of their method is different to ours, and our setting is more challenging. Note that the gradient update in \cite{lee2013efficient} is the same as the typical coordinate descent, which naturally requires only $\O(1)$ computation. What they try to avoid is the computation in the extrapolation step of the other coordinates. As a contrast, we do not only have extrapolation step, but the gradients used in update \eqref{eq:alg2_x} are also the sum of dense vectors. Due to such extra difficulty, our decomposing scheme is different and more complicated than the one in \cite{lee2013efficient}.
\end{remark}
}

{
\begin{remark}\label{rem:sparse-update}
We point out that a sparse implementation of stochastic SPDC was also proposed in \cite{zhang2017stochastic}, and similar idea for Prox-SVRG can be found in \cite{xiao2014proximal}. Such idea can also be extended to other stochastic methods like ProxSGD and SAGA. However, all these methods implicitly require
\[
\mathrm{prox}_g^{(t)}\coloneqq
\underbrace{\mathrm{prox}_g\circ \cdots \circ \mathrm{prox}_g(x)}_{\text{composition of $t$ proximal mappings}}
\]
can be easily computed in constant time independent of $t$. This property enables them to ignore the iterations with zero gradients and is the key for their sparse update trick. However, such property is only satisfied by some special $g(x)$, and only examples on simple regularizers $g(x)=\lambda\|x\|_1$ and $g(x)=(\lambda/2)\|x\|^2_2$ are given in their papers. For these two regularizers, it is quite easy to show that
\[
\mathrm{prox}_{g_j}^{(t)}(x_j) = \mathrm{prox}_{g_j}^{(t-1)}(x_j)\cdot \frac{1}{1+\lambda} =\cdots = \frac{x_j}{(1+\lambda)^t}
\]
for $g(x)=(\lambda/2)\|x\|_2^2$, and
\[
\mathrm{prox}_{g_j}^{(t)}(x_j)
= \left\{ \begin{array}{ll}
x_j - \mathrm{sign}(x_j)\cdot\lambda t \quad&\text{if $|x_j|\geq \lambda t$} \\
0 \quad&\text{otherwise}
\end{array}\right.
\]
if $g(x)=\lambda\|x\|_1$.
However, as far as we can see, it would be difficult to generalize their method to other regularizers such as KL-divergence, namely,
\[
g(x)=\sum_{j=1}^d w_j\log \frac{w_j}{x_j},
\]
which is commonly used in model-based transfer learning \cite{pan2009survey}. For this $g(x)$, computing a proximal mapping needs to solve a quadratic equation which does not admit a simple form of solution. Hence, it is hard to compute $\mathrm{prox}_g^{(t)}(x)$ without computing $\mathrm{prox}_g^{1}(x),\dots,\mathrm{prox}_g^{(t-1)}(x)$  one by one. As a result, their sparse update method would fail on such regularizer. As a comparison, our method does not rely on such assumption and works with any $g(x)$ as long as it is separable.
\end{remark}
}


\section{Numerical Experiments}
\label{sec:exp}

In this section, we conduct numerical experiments to DAPD and SDAPD and compare their performance with the following relevant existing methods:
\begin{itemize}
  \item {PDHG: primal-dual hybrid gradient method \cite{chambolle2011first}}
  \item APGM: Nesterov's accelerated proximal gradient method \cite{nesterov2013introductory}
  \item DA: original dual averaging method \cite{nesterov2009primal}
  \item RDA: regularized dual averaging method \cite{xiao2010dual}
  \item ProxSGD: proximal stochastic (sub-)gradient method \cite{shamir2013stochastic}
  \item ProxSVRG: proximal stochastic variance-reduced gradient method \cite{xiao2014proximal}
  \item SPDC: stochastic primal-dual coordinate method \cite{zhang2017stochastic}
\end{itemize}
Note that the first three methods are deterministic methods, while the others are stochastic methods. Besides, PDGH, APGM and SPDC are accelerated methods.
ProxSGD refers to proximal stochastic gradient descent, when we conduct experiments on smooth problems, and refers to stochastic subgradient method if it is applied to non-smooth problems.
{Though our analysis is based on the ergodic solutions $(\hat{x}^T, \hat{y}^T)$, we mainly report the behavior of the non-ergodic solutions. This is a common practice, because non-ergodic solutions preserve the solution sparsity. For completeness, we also report some comparison of the behavior of the ergodic and non-ergodic solutions in Figure \ref{fig:synthetic_ergodic}.}

\begin{figure*}[tb]
\centering
\includegraphics[width=0.32\linewidth]{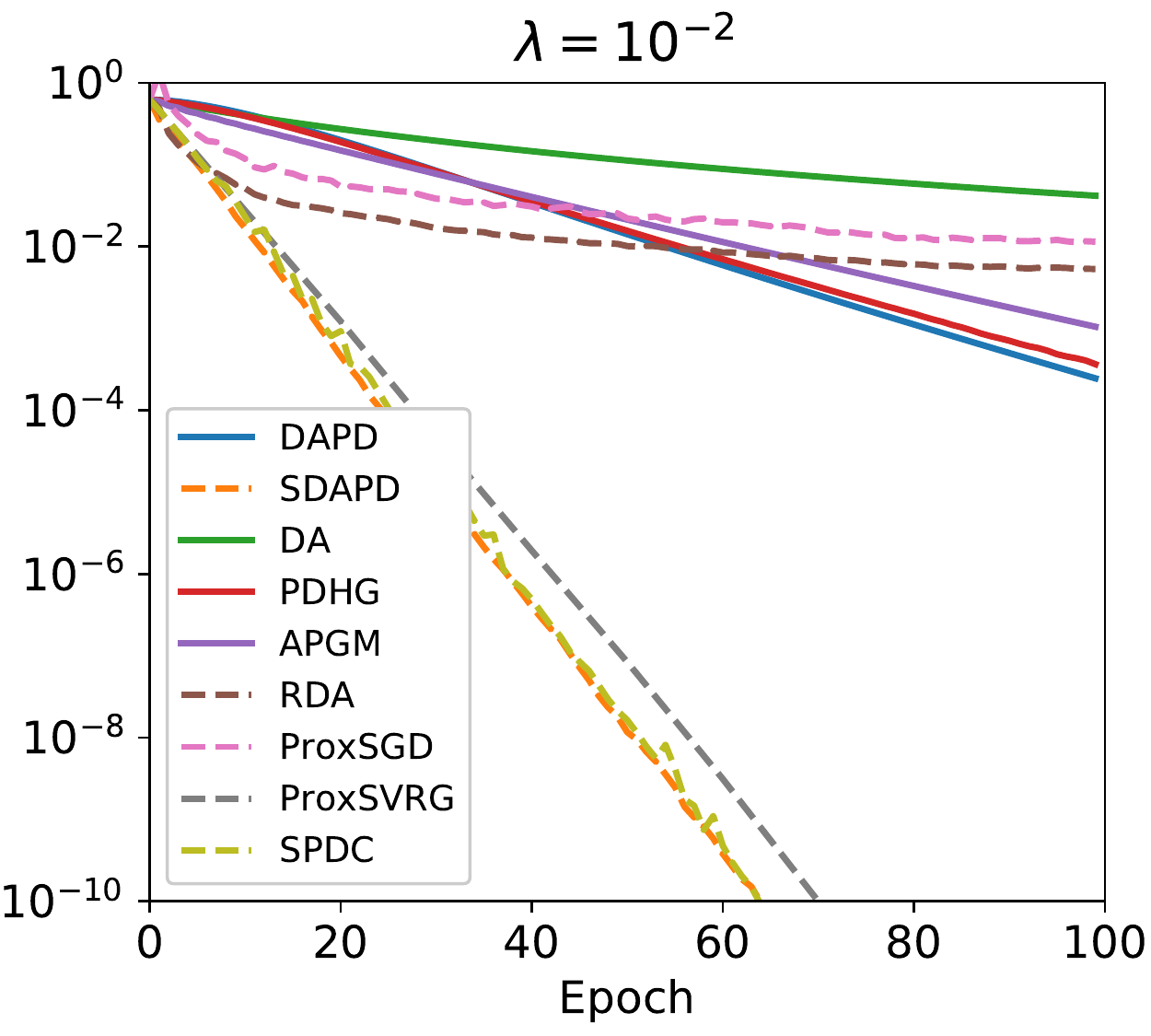}
\includegraphics[width=0.32\linewidth]{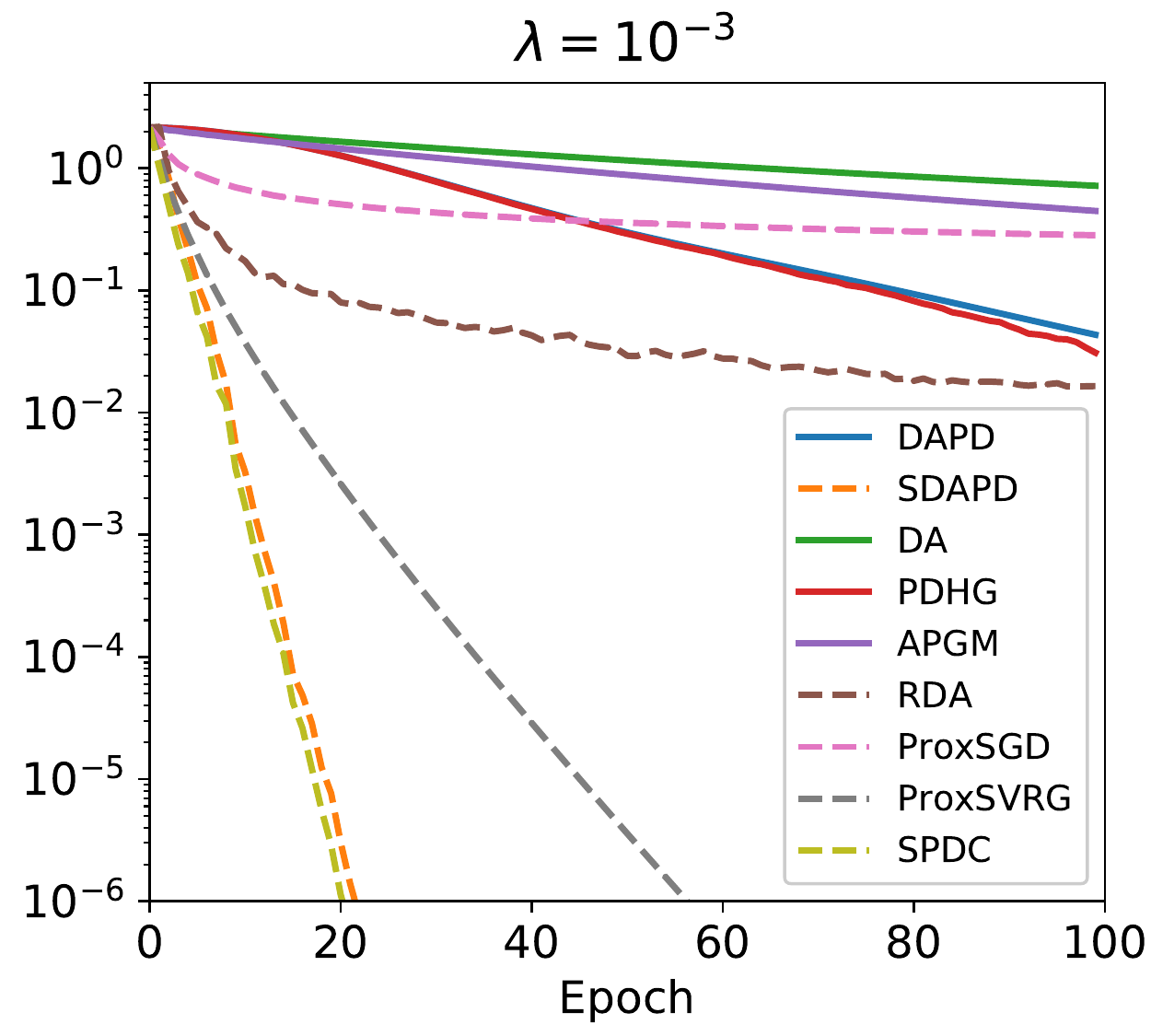}
\includegraphics[width=0.32\linewidth]{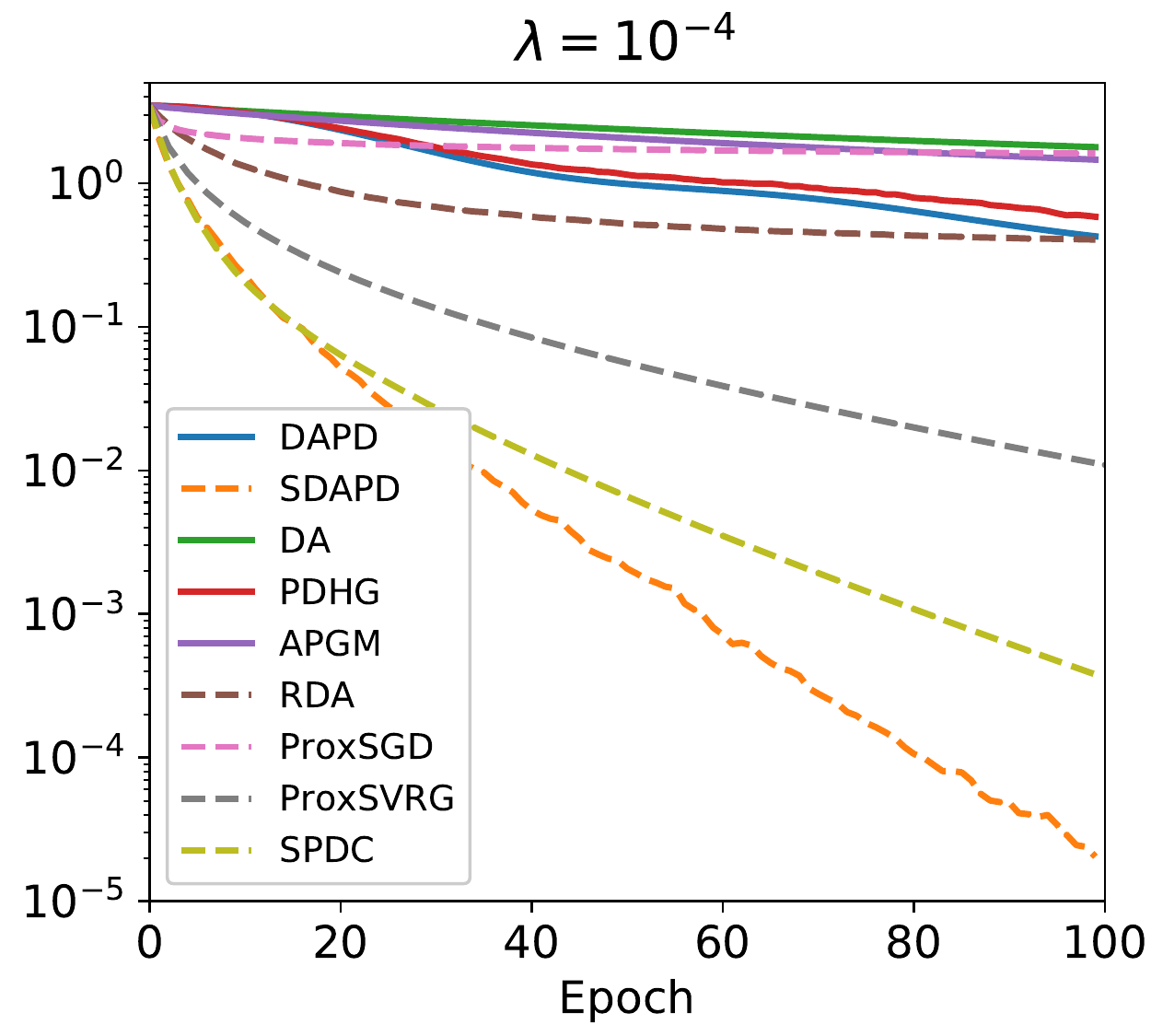}
\caption{Comparison on synthetic data with different choices of $\lambda$. The $y$-axis is the primal sub-optimality, namely $P(x^t) - P(x^*)$.
The solid lines are deterministic methods, and the dashed lines stand for stochastic methods. Here an epoch refers to one iteration for deterministic methods, and $n$ times accesses to the vectors $a_i$ for stochastic methods.}
\label{fig:synthetic}
\end{figure*}

\begin{figure*}[t]
\centering
\includegraphics[width=0.32\linewidth]{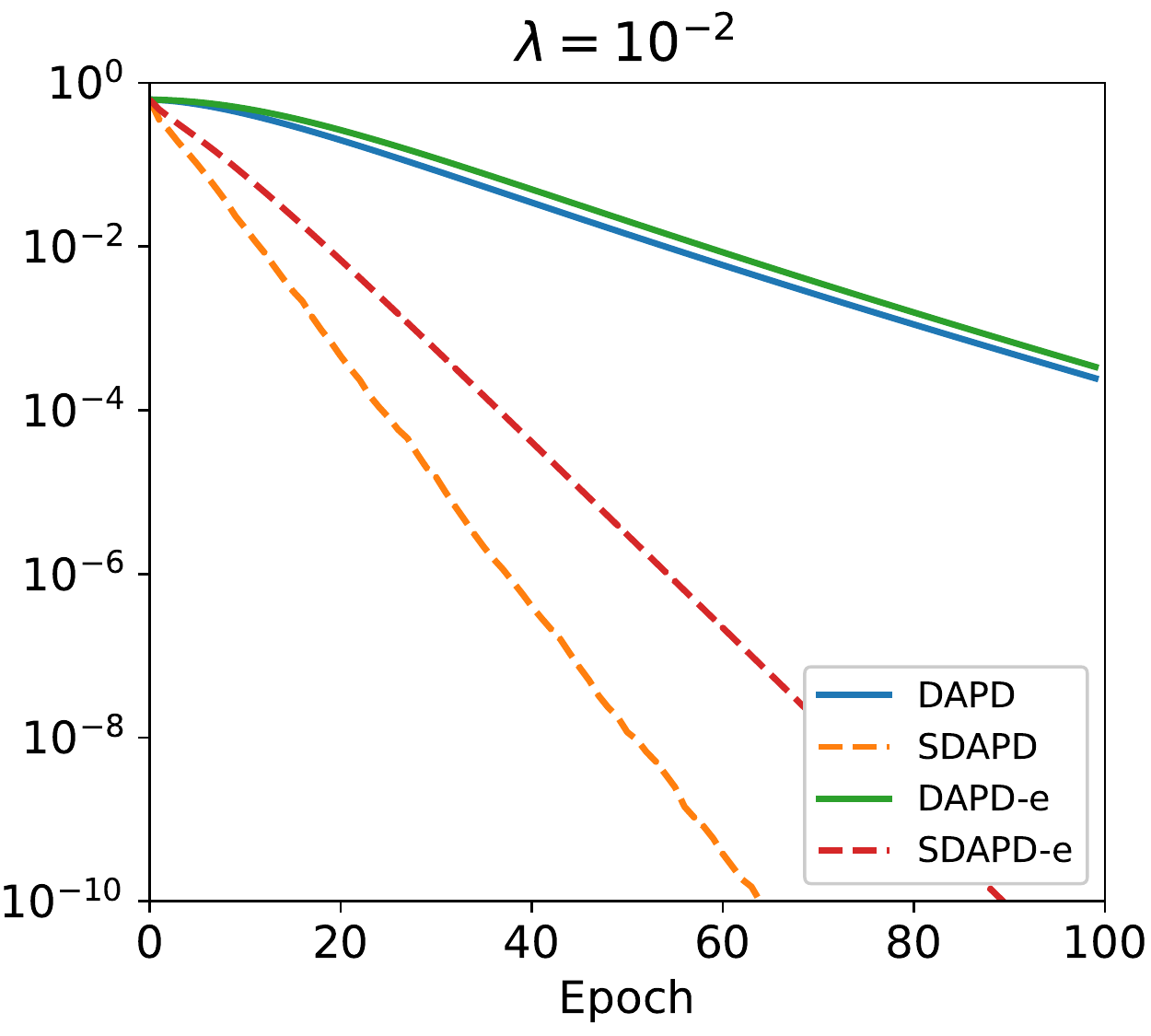}
\includegraphics[width=0.32\linewidth]{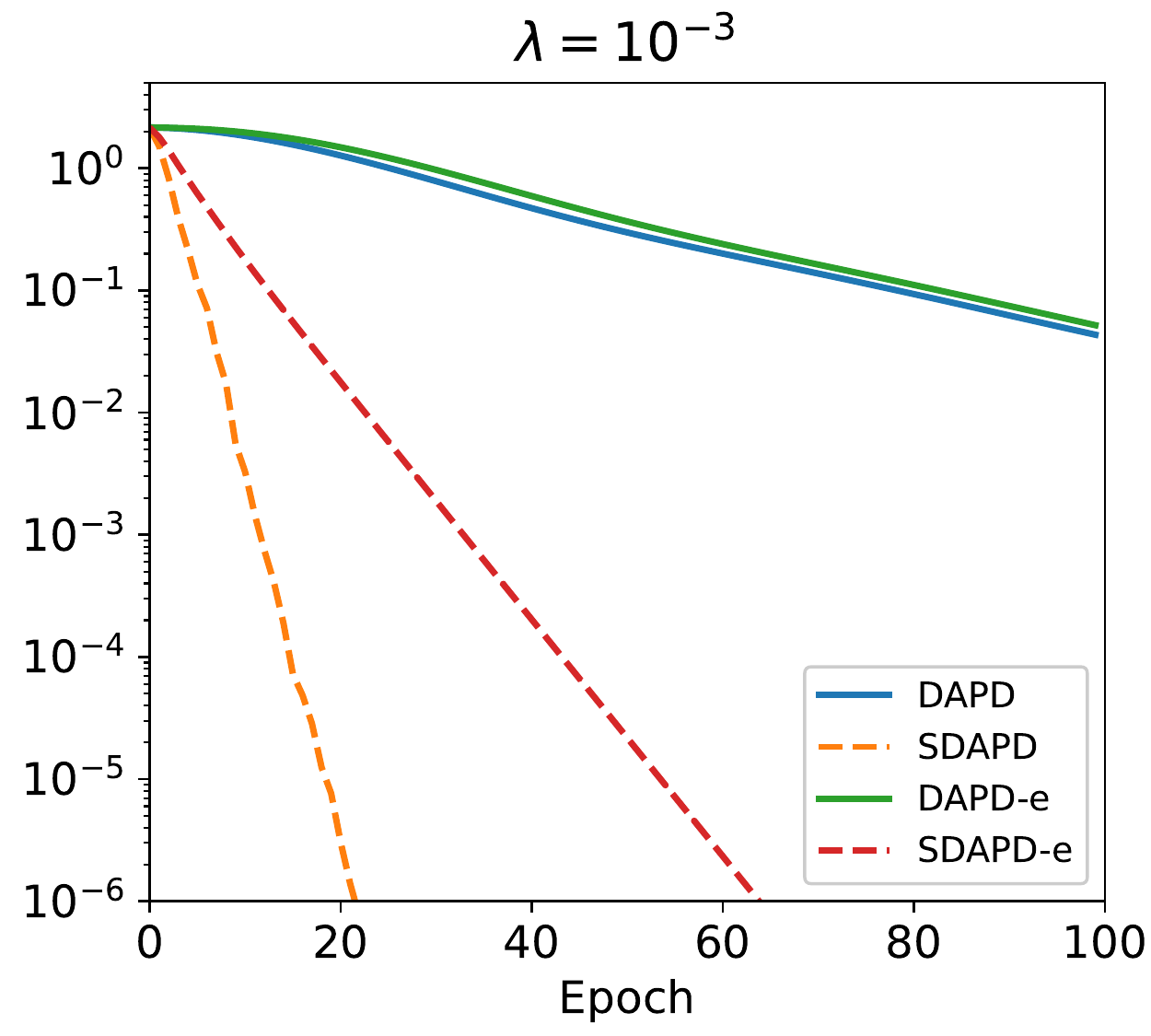}
\includegraphics[width=0.32\linewidth]{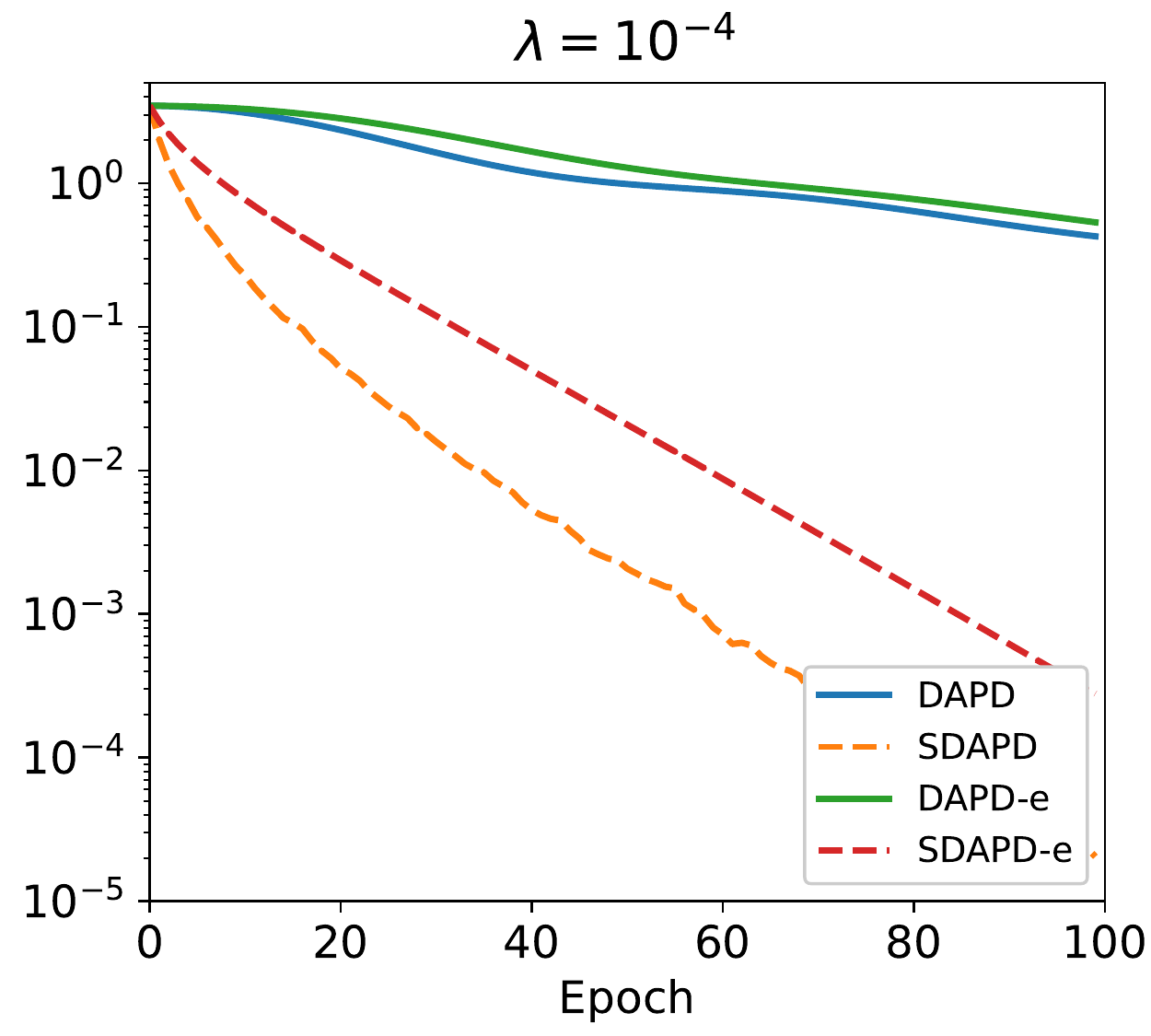}
\caption{{Comparison between ergodic and non-ergodic solutions on synthetic data. The $y$-axis is the primal sub-optimality, namely $P(x^t) - P(x^*)$. Lines with suffix ``-e'' stand for ergodic solutions, while others are non-ergodic.}}
\label{fig:synthetic_ergodic}
\end{figure*}

\subsection{Ridge Regression on Synthetic Data}
First, we test these algorithms on a ridge regression problem:
\[
\min_{x\in\R^d} \frac{1}{n}\sum_{i=1}^n \frac{1}{2}\left(\ip{a_i,x} - b_i\right)^2 + \frac{\lambda}{2}\|x\|^2_2
\]
with $\lambda>0$. Note that this problem is smooth and $\lambda$-strongly convex. We use synthetic data for this problem. Specifically, we first randomly generate a $x^*\in\R^d$, then each $a_i$ and $b_i$ are independently draw from the following model:
\[
b_i=\ip{x^*,a_i}+\varepsilon_i\text{ with } a_i\sim\mathcal{N}(0,\Sigma)
\text{ and } \varepsilon_i\sim\mathcal{N}(0,\sigma^2)
\]
{for some pre-chosen covariance matrix $\Sigma\in\R^{d\times d}$ and constant $\sigma>0$.}
In this experiment, we choose $n=d=1000$. We test the algorithms for different $\lambda$, which controls the condition number of the problem.
Note that smaller $\lambda$ leads to larger condition number.

The experiment results are presented in Figure \ref{fig:synthetic}. In all three sub-figures, the performances of SDAPD and SPDC are quite close,
and are always better than all other methods. Besides, when $\lambda=10^{-2}$, ProxSVRG also performs well, but it soon becomes inferior than SDAPD and SPDC when $\lambda$ gets smaller, since ProxSVRG is not an accelerated method. We also found that DAPD method performs similarly as PDGH, and is always much better than the other two deterministic methods: APGM and DA, though it is slower than some stochastic methods. Besides, when the condition number becomes larger, the performance difference between deterministic methods and stochastic methods becomes more prominent.

{Although the ergodic solutions of dual-averaging-type methods are rarely used in practice, we still report its behavior in Figure \ref{fig:synthetic_ergodic} for completeness. Figure \ref{fig:synthetic_ergodic} shows that the ergodic solutions converge slower than the non-ergodic solutions.}

\begin{figure*}[t]
\centering
\includegraphics[width=0.32\linewidth]{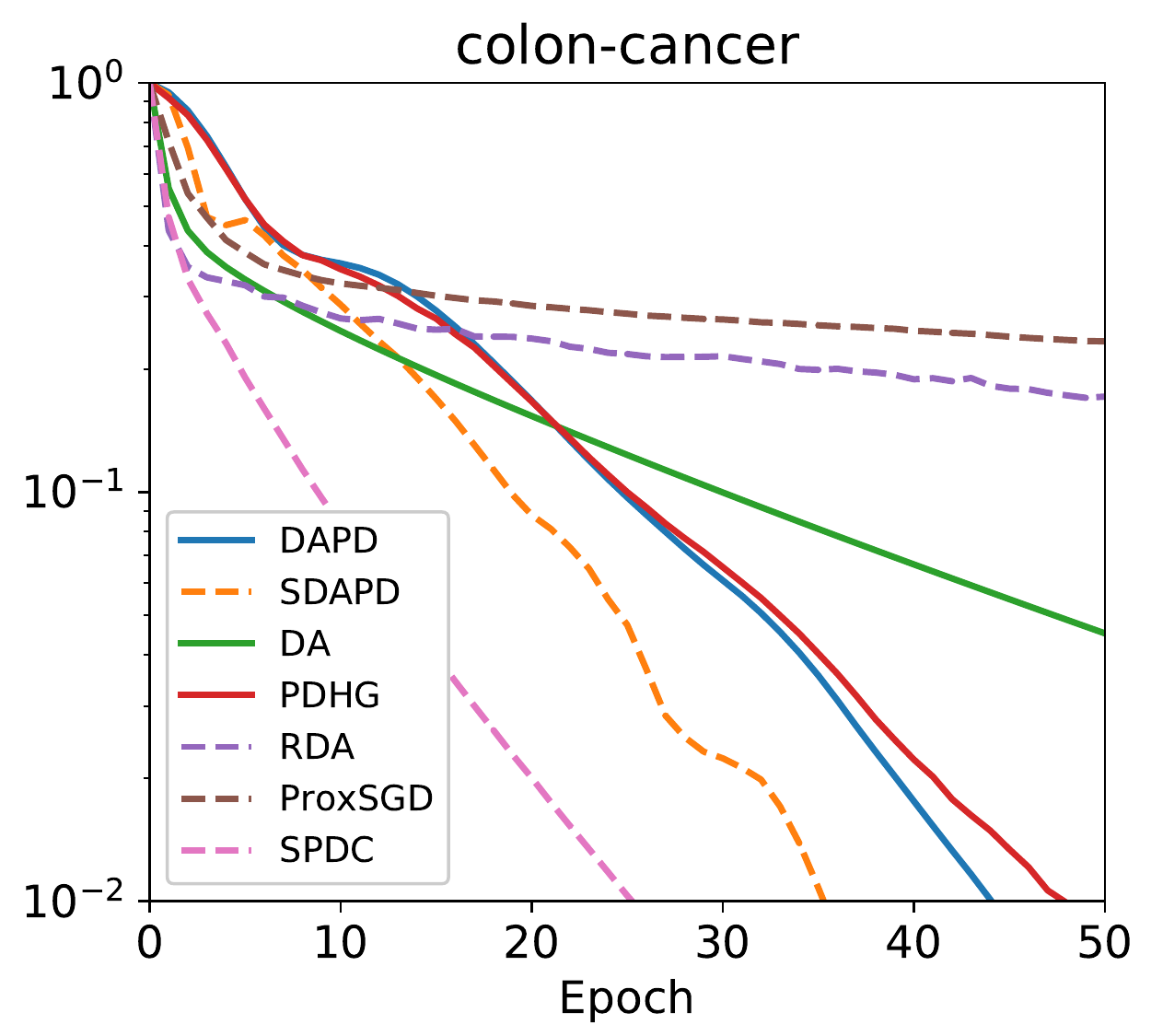}
\includegraphics[width=0.32\linewidth]{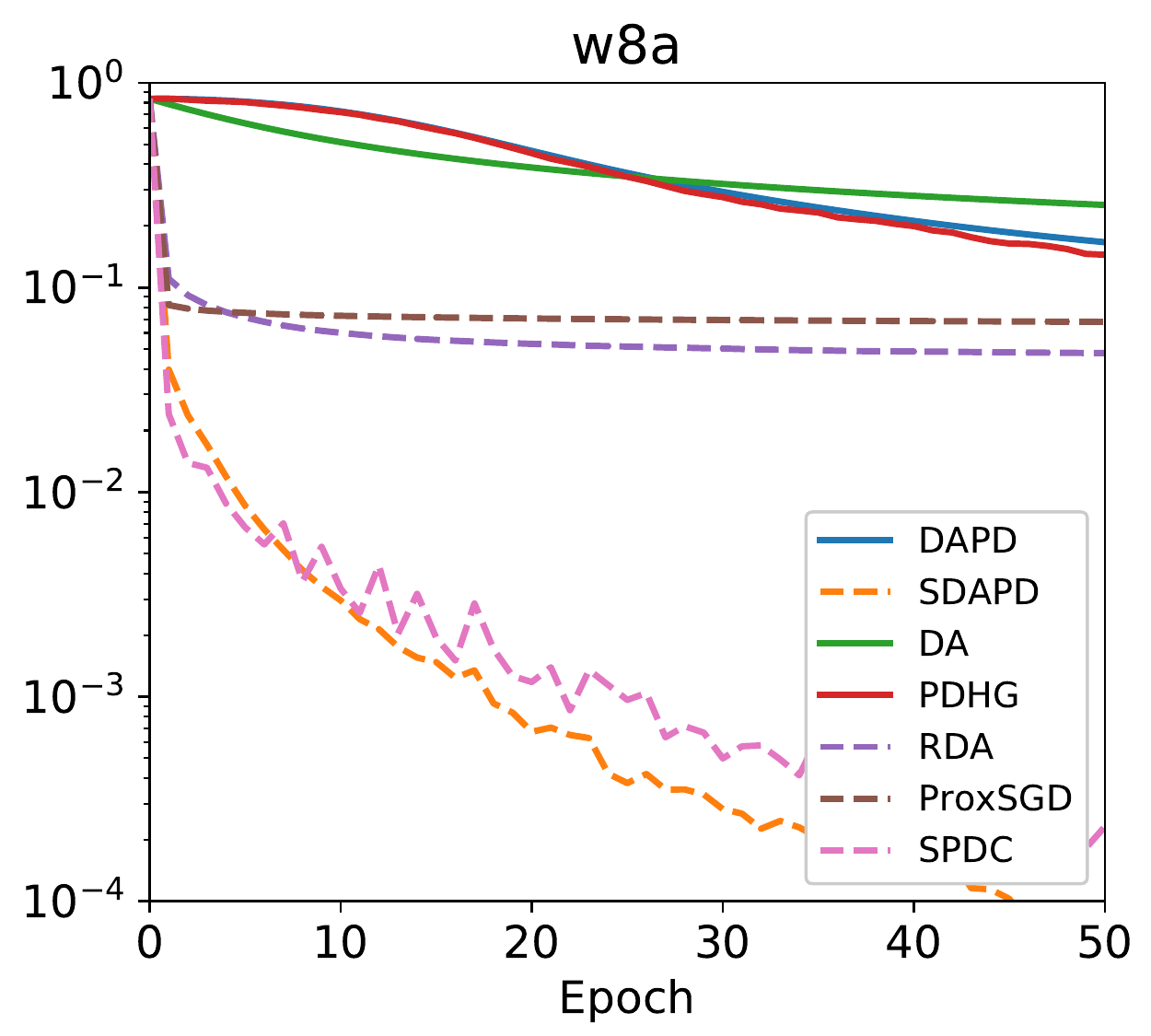}
\includegraphics[width=0.32\linewidth]{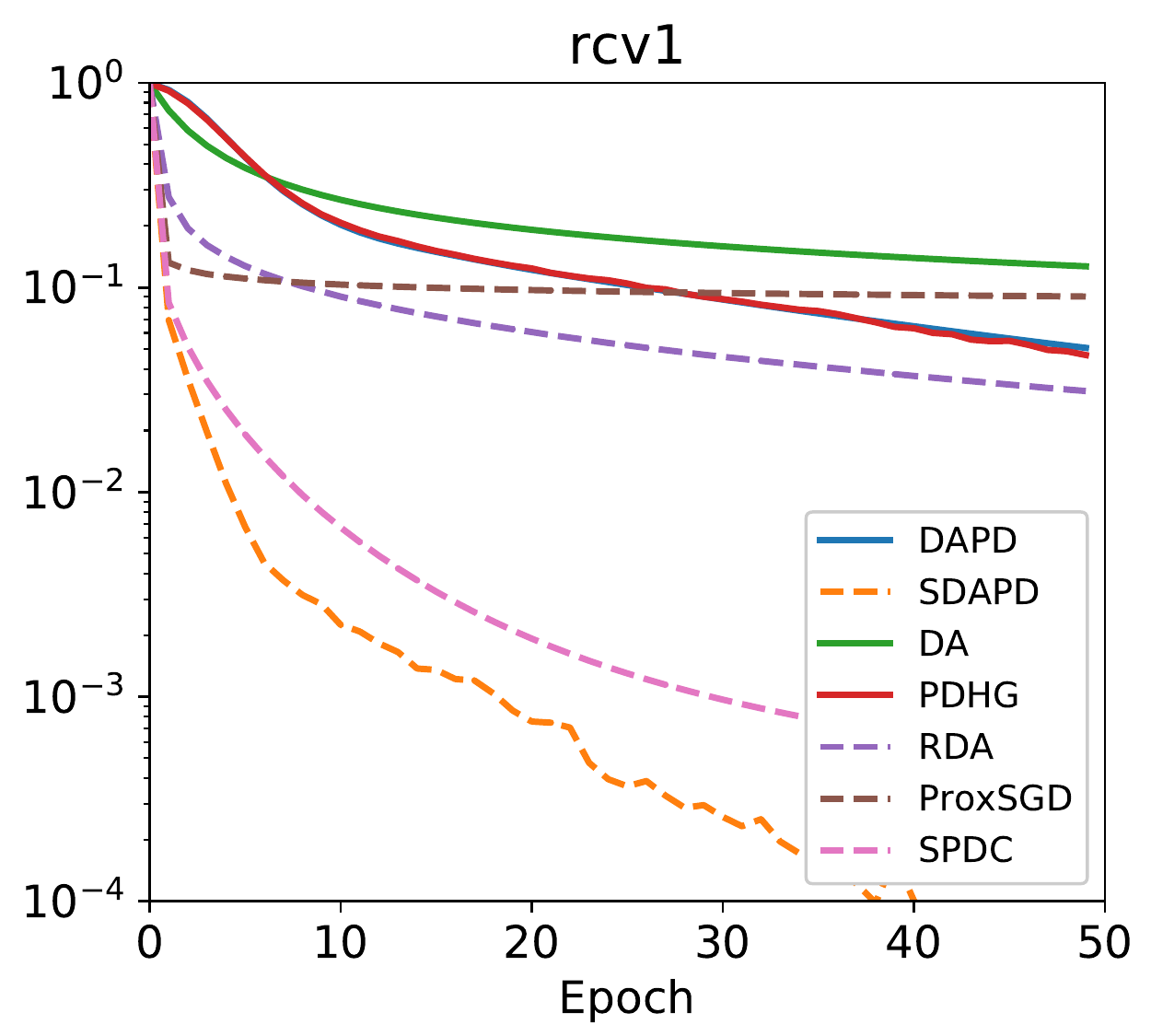}
\caption{Results on real datasets. The $y$-axis is the primal sub-optimality.}
\label{fig:real}
\end{figure*}

\subsection{Classification via SVM on Real Datasets}
\label{sec:exp_svm}

\begin{table}[b]
\centering
\caption{\centering Summary of datasets\\ }
\label{table:dataset}
\begin{tabular}{|c|c|c|c|}
\hline
Dataset & $n$ & $d$ & $\rho$ \\ \hline\hline
\texttt{colon-cancer} & 62 & 2,000 & $100\%$ \\ \hline
\texttt{w8a} & 49,749 & 300 & $3.88\%$ \\ \hline
\texttt{rcv1} & 20,242 & 47,236 & $0.16\%$ \\ \hline
\end{tabular}
\end{table}

{In this part, we test the algorithms on the binary classification task via support vector machine (SVM):
\[
\min_{x\in\R^d}\frac{1}{n}\sum_{i=1}^n \max\left\{1-\ip{b_i a_i, x},\, 0\right\} + g(x).
\]
Here we choose $g(x)$ to be the Huber's regularization, which is defined as: $g(x)=\sum_{j=1}^d g_j(x_j)$ with
\begin{equation}
\label{eq:huber}
g_j(x_j)=\left\{ \begin{array}{ll}
\lambda\big(|x_j| - \frac{\lambda}{4\mu}\big)\quad&\text{if }|x_j|\geq \frac{\lambda}{2\mu}, \\
\mu x_j^2\quad &\text{otherwise}.
\end{array}\right.
\end{equation}
Huber's regularization can also help the model to avoid over-fitting just like the squared-$\ell_2$-norm, but it is statistically more robust than the latter one \cite{zadorozhnyi2016huber}. As far as we know, it would be hard for ProxSGD and SPDC to have sparse update with such $g(x)$. For this experiment, we fix the parameters as $\lambda=10^{-4}$ and $\mu=1$ in Huber's regularization.}
Besides, it should be noted that our objective function is non-smooth in this case. Since APGM and ProxSVRG are unable to deal with such kind of objective, they are not tested for this problem. We use real datasets in this experiment. The dataset information is summarized in Table \ref{table:dataset}. \texttt{w8a} and \texttt{rcv1} are sparse datasets.

The experiment results are presented in Figure \ref{fig:real}. The results are similar to the ones for ridge regression. We observe that SDAPD performs better than all other methods, except that it falls behind SPDC on \texttt{colon-cancer}. Again, the performances of DAPD and PDHG are very close, but they are much better than the other deterministic method DA. Only thing interesting to note here is that the performance of DAPD is close to SDAPD on \texttt{colon-cancer}. It is because this dataset has a relatively small $n$, thus deterministic methods and stochastic methods do not make too much difference in their convergence rates.

We also report the per-epoch running time of each algorithm in Table \ref{table:time}. We see that deterministic methods DAPD, PDGH and DA are always the fastest, since they can do updates in batch with highly-optimized matrix-vector operations. We can also find that ProxSGD and SPDC are quite time-consuming on \texttt{w8a} and \texttt{rcv1} datasets because they are unable to do sparse updates, while our SDAPD overcomes this issue with the help of the sparse update strategy introduced in Section \ref{sec:sparse} and therefore has much less computational cost on sparse data. However, such strategy requires to maintain some auxiliary variables, resulting more computational time to SDAPD than RDA, and even costs more time than ProxSGD and SPDC on dense data. Of course, SDAPD can be further improved by discarding the sparse update strategy on dense data. But we do not adopt this here. Instead, we implement SDAPD in a uniform way to keep the experiments simple.

Overall, by taking both convergence rate and per-epoch computation time into account, SDAPD is the best one among all tested algorithms.

{

\begin{table}[t]
\centering
\caption{\centering Per-epoch running time of each method in seconds\\}
\label{table:time}
\begin{tabular}{|c|c|c|c|}
\hline
Methods & \begin{tabular}{@{}c@{}}\texttt{colon-cancer} \\ ($\times 10^{-3}$) \end{tabular}
 & \begin{tabular}{@{}c@{}}\texttt{w8a} \\ ($\times 10^{-2}$) \end{tabular}
 & \begin{tabular}{@{}c@{}}\texttt{rcv1} \\ ($\times 10^{-2}$)\end{tabular} \\ \hline\hline
DAPD & 1.0 & 5.3 & 4.2 \\ \hline
SDAPD & 5.2 & 17.5 & 19.8 \\ \hline
PDHG & 1.0 & 5.7 & 4.0 \\ \hline
DA & 1.0 & 5.5 & 3.6 \\ \hline
RDA & 2.7 & 7.4 & 6.0 \\ \hline
ProxSGD & 2.9 & 30.3 & 1932 \\ \hline
SPDC & 2.7 & 46.3 & 2125 \\ \hline
\end{tabular}
\vspace{-0.2cm}
\end{table}

\begin{figure*}[t]
\centering
\includegraphics[width=0.32\linewidth]{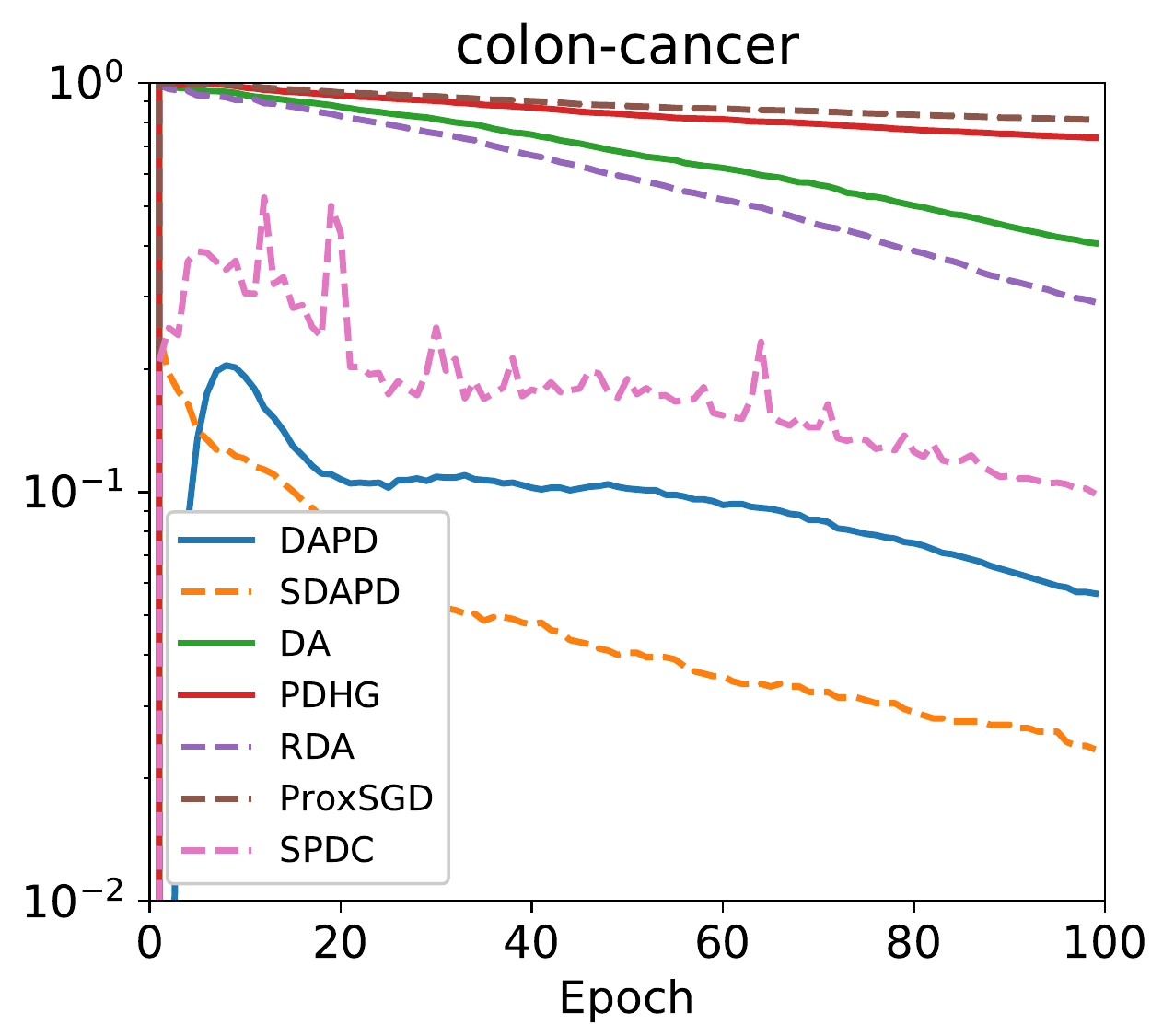}
\includegraphics[width=0.32\linewidth]{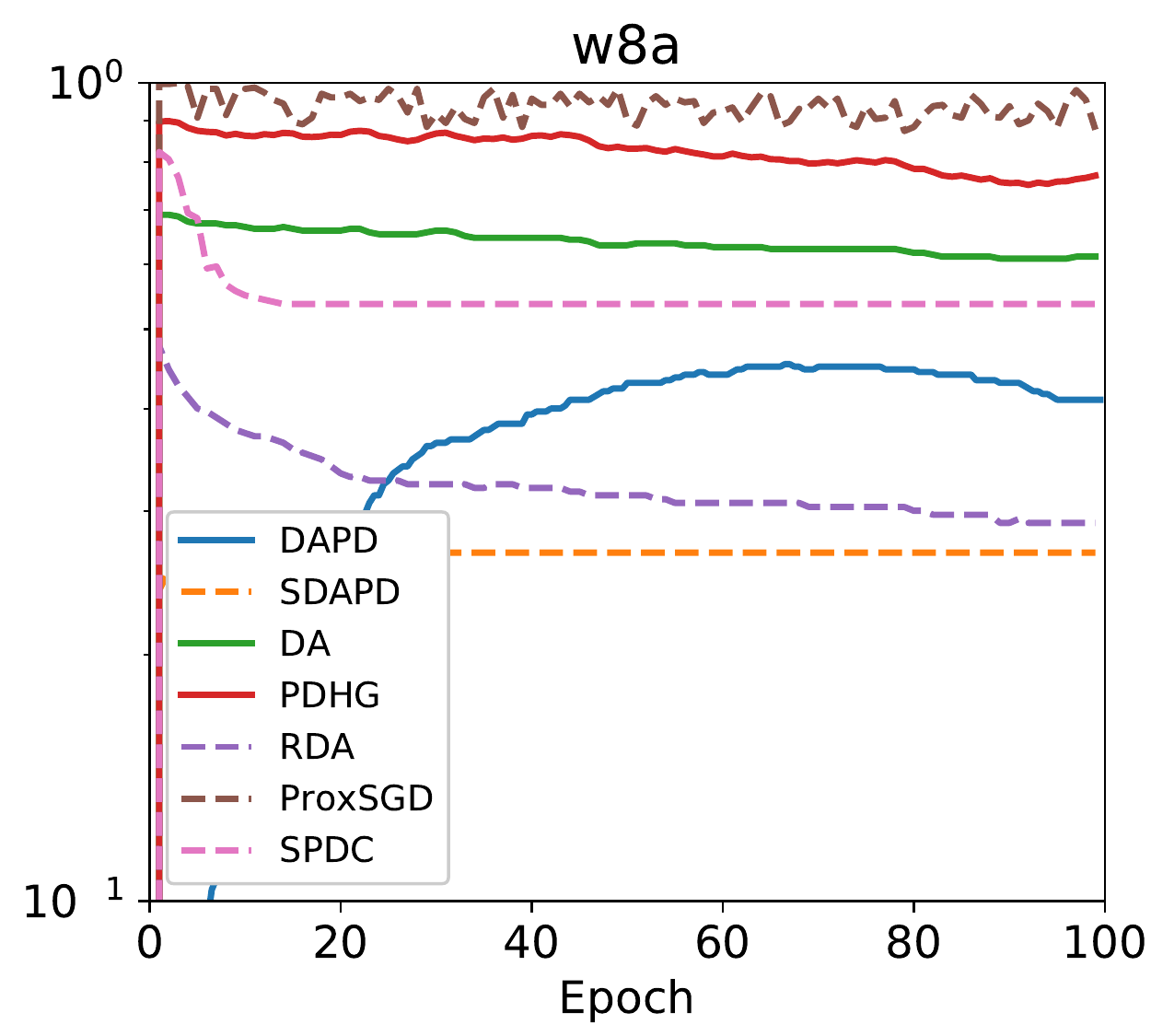}
\includegraphics[width=0.32\linewidth]{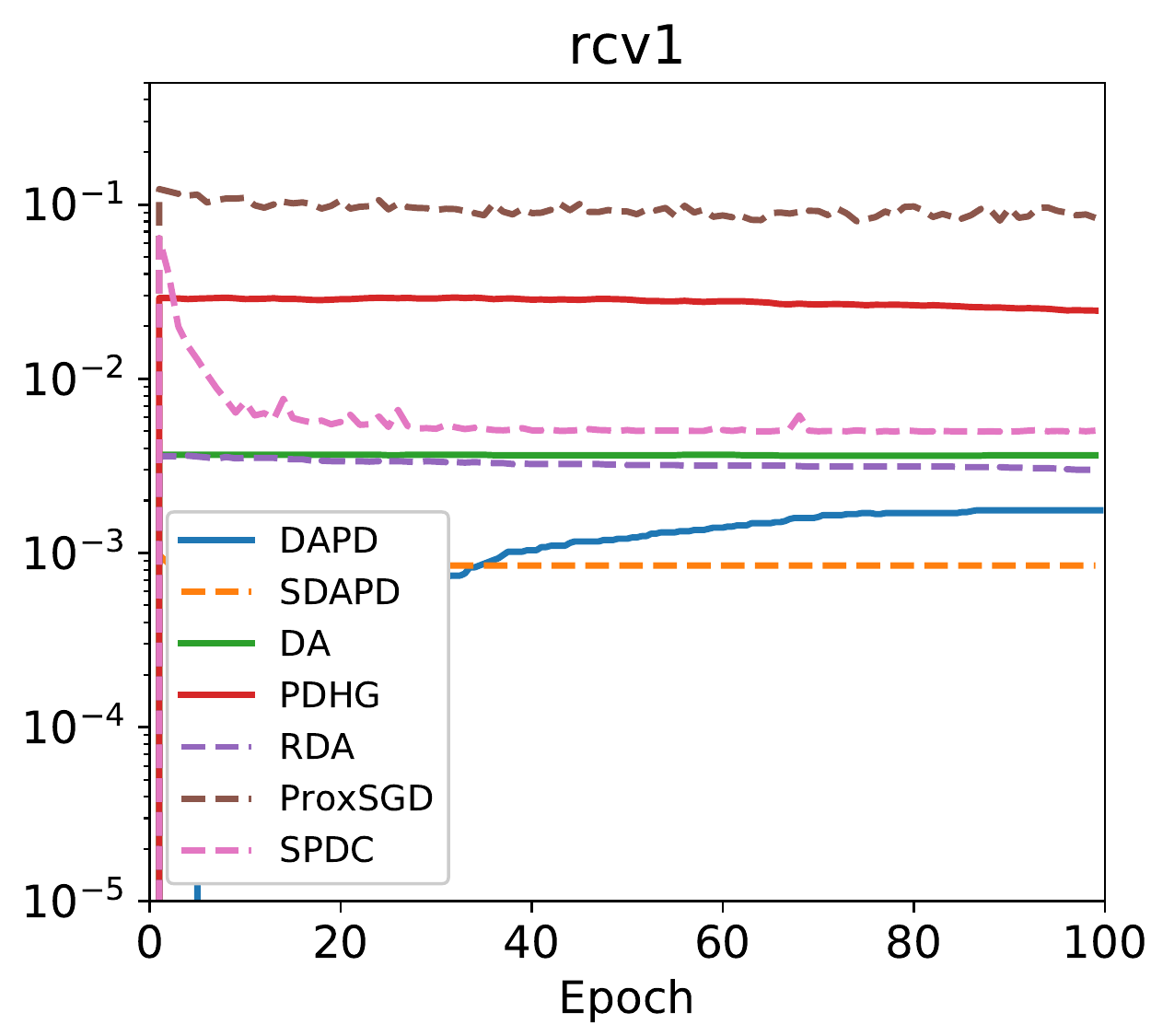}
\caption{{Proportion of non-zeros in the generated non-ergodic solutions.}}
\label{fig:sparsity}
\end{figure*}

\subsection{Comparison on Solution Sparsity}
In this part, we focus on the same setting as the previous part. However, we change the regularizer to $g(x)=\lambda\|x\|_1$ to induce sparse solution, so that we can observe the influence of different optimization methods to solution sparsity. Again, we fix $\lambda=10^{-4}$ on all the datasets.

The results are presented in Figure \ref{fig:sparsity}, which show that our DAPD and SDAPD can produce sparser solutions than most baselines. The only exception is RDA, which is comparable with DAPD and SDAPD on the \texttt{w8a} dataset. This is expected because RDA is known to promote solution sparsity. Moreover, we found that stochastic algorithm SDAPD always outperforms RDA on the tested problems. We conjecture it is because of the increasing primal step sizes $\{\beta_t\}$ in our algorithms that make regularization effects even stronger.
}

\section{Conclusion}
In this paper, we proposed a dual-averaging primal-dual method (DAPD), which combines the idea of dual averaging and primal-dual method, and can solve a wide range of optimization problems with composite convex objective. Our analysis shows that DAPD has optimal convergence rates in several different settings. We also proposed a stochastic version of DAPD (SDAPD) for solving convex problems with a finite-sum objective. A novel way is proposed to efficiently implement SDAPD for sparse data. We demonstrated the superiority of our methods by comparing them with several existing methods on standard machine learning tasks.

\section*{Acknowledgements}

The authors are grateful to two anonymous referees for providing insightful and constructive comments that greatly improved the presentation of this paper. The research of S. Ma was supported in part by a startup package in the Department of Mathematics at University of California, Davis.

\bibliography{reference,NSF-nonsmooth-manifold}

\end{document}